\documentclass[a4paper,reqno]{amsart}

\usepackage{amsmath}
\usepackage{amssymb}
\usepackage{amsfonts}
\usepackage{amsthm}
\usepackage{mathrsfs}
\usepackage{mathtools}
\usepackage[colorlinks,unicode,psdextra]{hyperref}
\usepackage{enumerate}
\usepackage{tikz}
\usepackage{subcaption}

\mathtoolsset{showonlyrefs=true}
\numberwithin{equation}{section}
\allowdisplaybreaks

\newtheorem{thm}{Theorem}[section]
\newtheorem{lem}[thm]{Lemma}
\newtheorem{prop}[thm]{Proposition}
\newtheorem{coro}[thm]{Corollary}
\newtheorem{defi}[thm]{Definition}
\newtheorem{remark}[thm]{Remark}

\newcommand{\me}{\mathrm{e}}
\newcommand{\mi}{\mathrm{i}}
\newcommand{\dif}{\mathop{}\!\mathrm{d}}
\newcommand{\cone}{\mathrm{Cone}}
\newcommand{\diam}{\operatorname{diam}}
\newcommand{\crf}{\mathfrak{c}}

\begin{document}
\title[Sharp hyperbolic Strichartz estimate]{On sharp Strichartz estimate for hyperbolic Schr\"{o}dinger equation on $\mathbb{T}^3$}

\author{Baoping Liu}
\address{Department of Mathematics, School of Mathematical Sciences\\
Peking University\\
Beijing\\ China}
 \email{baoping@math.pku.edu.cn}

  \author{Xu Zheng}
 \address{Department of Mathematics, School of Mathematical Sciences\\
Peking University\\
Beijing\\ China}
 \email{xuzheng-math@stu.pku.edu.cn}

\thanks{
The authors are supported by the NSF of China (No. 12571254, 12341102).}
\begin{abstract}
    We prove the sharp Strichartz estimate for hyperbolic Schr\"{o}dinger equation on $\mathbb{T}^3 $  via an incidence geometry approach. As application, we obtain  optimal local well-posedness of nonlinear hyperbolic Schr\"{o}dinger equations.

\end{abstract}
\maketitle

\section{Introduction}
The question of Strichartz estimates for Schr\"{o}dinger equation on tori was first addressed by Bourgain~\cite{Bourgain1993}.
 Later, Bourgain-Demeter~\cite{bour2015} proved the full range Strichartz estimates with  $N^\varepsilon$ loss by the Fourier decoupling method:
 \begin{equation}\label{Bourgain-Demeter Strichartz estimate by decoupling}
     \| P_N \me^{\mi t\Delta} \phi \|_{{L_{t,x}^p([0,1]\times \mathbb T^d)}} \lesssim_\varepsilon N^{\frac{d}{2} - \frac{d+2}{p}+\varepsilon } \| \phi \|_{L^2_x(\mathbb T^d)}, \quad \forall p\geq p_{d},\ \forall \varepsilon>0,
 \end{equation}
 where $P_N$ denotes the Littlewood-Paley projection operator to frequency $N$ and $p_d = 2(d+2)/d$.  
 Decoupling theorems are powerful and robust tools in Fourier analysis, but the $N^\varepsilon$ loss is inherent in the proof of decoupling theorems.
 The  loss in \eqref{Bourgain-Demeter Strichartz estimate by decoupling} was  removed by Killip-Vi\c{s}an~\cite{Killip2016} for $p>p_d$. Recently, Herr-Kwak~\cite{Herr-Kwak} proved the sharp  endpoint point $L^4$ estimate on $\mathbb T^2$  
 \begin{equation}
     \| P_N \me^{\mi t\Delta} \phi \|_{L^4_{t,x}([0,1]\times \mathbb T^2)} \lesssim (\log N)^{1/4} \| \phi \|_{L^2_x(\mathbb T^2)},
 \end{equation}
which implies global existence of solutions to the cubic
 (mass-critical) nonlinear Schrödinger equation in $H^s(\mathbb T^2)$ for any $s>0$.

For the hyperbolic Schr\"{o}dinger equation, it shares the same Strichatz estimates as the elliptic one in the Euclidean case, but there is a difference on tori. In \cite{bourgain2017}, Bourgain-Demeter proved that 
  \begin{equation}
     \| P_N \me^{\mi t\Box} \phi \|_{L^p_{t,x}([0,1]\times \mathbb T^d)} \lesssim_\varepsilon N^{\mu_{d,v}(p)+\varepsilon}
     \| \phi \|_{L^2_x(\mathbb T^d)}, \quad \forall p\geq 2,\ \forall \varepsilon>0, \label{est:bourgain}
 \end{equation}

 where $\Box = \partial^2_{x_1}+\dots+\partial^2_{x_v}-\partial^2_{x_{v+1}}-\dots-\partial^2_{x_d}$,  $v\leq d/2,$ and 
 \begin{equation}
 \mu_{d,v}(p) = \max\left\{ \frac{d}{2} - \frac{d+2}{p}, \frac v2-\frac vp\right\}.\label{powerindex} 
 \end{equation}
 The factor $N^{v(\frac12-\frac1p)}$ is due to that the hyperbolic paraboloid contains a vector subspace of dimension $v$. It's a natural question to ask whether the $N^\varepsilon$ loss can be removed.

 In this paper, we consider the case $d=3$ and  prove the sharp Strichartz estimate for hyperbolic Schr\"{o}dinger equation without $N^\varepsilon$ loss.  With the notations {$\Box = \partial^2_{x_1} - \partial^2_{x_2} - \partial^2_{x_3}$} and $\mu(p)=\mu_{3,1}(p)=\max\{\frac{3}{2}-\frac{5}{p}, \frac12-\frac{1}{p}\}$, our main result reads as
follows:
 \begin{thm}\label{sharp T3 L4 hyperbolic Strichartz}
     For $\phi \in L^2(\mathbb T^3)$, we have that
     \begin{equation}
     \| P_N \me^{\mi t\Box} \phi \|_{L^p_{t,x}([0,1]\times \mathbb T^3)} \lesssim
     N^{\mu(p)} \| \phi \|_{L^2_x(\mathbb T^3)},\quad \forall p\geq 2. \label{main-est}
 \end{equation}
 \end{thm}

 \begin{remark}\label{remark:main}
     By invoking interpolation with $L^\infty$ and $L^2$ estimates, it suffices to prove Theorem~\ref{sharp T3 L4 hyperbolic Strichartz} for $p=4$, i.e.
      \begin{equation} \label{est:L4}
         \| P_N \me^{\mi t\Box}{\phi}\|_{{L_{t,x}^4([0,1]\times \mathbb T^3)}} \lesssim N^{1/4} \| \phi \|_{{L^2_x(\mathbb T^3)}}.
     \end{equation}
     \end{remark}

 
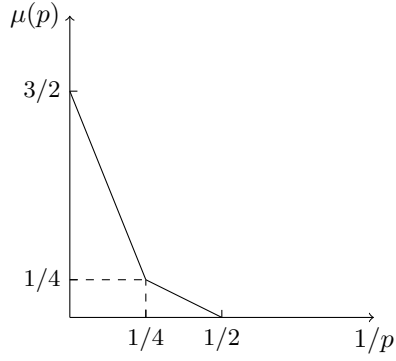
\begin{figure}[htbp]
 \begin{tikzpicture}[xscale=4,yscale=2]
 \draw[->] (0,0) -- (1,0);
 \draw[->] (0,0) -- (0,2);
 \draw (0,3/2) -- (1/4,1/4);
 \draw (1/4,1/4) -- (1/2,0);
 \draw[dashed] (0,1/4) -- (1/4,1/4);
 \draw[dashed] (1/4,0) -- (1/4,1/4);
 \node[below] at (1,0) {$1/p$};
 \node[left] at (0,2) {$\mu(p)$};
 \foreach \x in {1/4,1/2}{\draw (\x,0) -- (\x,0.05); \node[below,font=\small] at (\x,0) {\x};}
 \foreach \y in {1/4,3/2}{\draw (0,\y) -- (0.025,\y); \node[left,font=\small] at (0,\y) {\y};}
 \end{tikzpicture}

 \label{figure: relation between power index and p}
 \end{figure}
      \begin{remark}\label{rem:main-2}
   Due to the Galilean invariance of solutions
to the linear hyperbolic Schr\"{o}dinger equation, estimate  \eqref{est:L4} can be reformulated as 
  \begin{equation}
     \| P_S \me^{\mi t\Box} \phi \|_{L^4_{t,x}([0,1]\times \mathbb T^3)} \lesssim 
     \diam(S)^{1/4} \| \phi \|_{L^2_x(\mathbb T^3)}
 \end{equation}
 for any bounded set $S\subset \mathbb{Z}^3$.  Naively, we ask if this can be replaced by a bound depending only on $\# S$, as in the work of Herr-Kwak~\cite{Herr-Kwak}. We will construct  examples in Section~\ref{subsec:examples} to show there is no efficient bound except for the trivial one $(\#S)^{1/4}$.   

\end{remark}

As application of Theorem~\ref{sharp T3 L4 hyperbolic Strichartz}, we consider the Cauchy problem for  hyperbolic nonlinear Schr\"{o}dinger equations (HNLS).  HNLS arise in many physics contexts, such as plasma waves~\cite{Plasma-1,Plasma-2,Plasma-3, Plasma-4} and gravity water waves~\cite{Craig, DS, Saut,Totz2015cmp}. In particular, the 3d cubic HNLS appear in the study of optical self-focusing of short light pulses in nonlinear media~\cite{Plasma-4}, and it is considered one of the canonical NLS equations in 3d~\cite{ZK}.
We refer the readers to the survey paper by Saut-Wang~\cite{Saut2024hnls} for more details.

 The Cauchy problem of two-dimensional  periodic HNLS with cubic nonlinearity
 \begin{equation}\label{cubic HNLS on T2}
     (\mi\partial_t + \partial_{x_1}^2-\partial_{x_2}^2 ) u = |u|^2u,\quad (t,x) \in \mathbb R \times \mathbb{T}^2,
 \end{equation}
 has been considered by Godet-Tzvetkov~\cite{Tzvetkov} and Wang~\cite{wang2013}. They both established $L^4$ Strichartz estimate with $1/4$-derivative loss, using different methods. Besides, Wang~\cite{wang2013} used the Strichartz estimate to prove that the Cauchy problem of \eqref{cubic HNLS on T2} is locally well-posed in $H^{s}(\mathbb T^2)$ for $s>1/2$ while it's ill-posed for $s<1/2$ {in
 the sense that the solution map is not $C^3$ continuous in $H^s(\mathbb T^2)$ even for small data}.  The recent work \cite{Wang2025-1} established the sharp unconditional well-posedness in Fourier–Lebesgue spaces (modulo the endpoint case) for \eqref{cubic HNLS on T2} and~\cite{Wang2025-2} considered HNLS with all odd power nonlinearities on $\mathbb{R}\times\mathbb{T}$ and proved sharp local well-posedness.
 
Here we study the  three-dimensional periodic HNLS 
 \begin{equation}\label{HNLS on T3}
     \mi\partial_t u + \Box u = \pm |u|^{2k}u,\quad (t,x) \in \mathbb R \times \mathbb{T}^3,
 \end{equation}
 where $k$ is a positive integer. The Cauchy problem  for  \eqref{HNLS on T3}  was posed by Saut-Wang {\cite{Saut2024hnls}.}
 In the Euclidian case, the equation \eqref{HNLS on T3} enjoys the scaling symmetry, which leaves the critical Sobolev norm $\| \cdot\|_{\dot H^{s_c}(\mathbb R^3)}$ invariant for $s_c=\frac 32-\frac 1k$. Although in the periodic case we don't have this natural scaling symmetry, the notation of critical Sobolev index provides us heuristics. 
 We have the following results of local well-posedness.
 \begin{thm}\label{local well-posedness of HNLS on T^3}
     For $k\geq2$, the Cauchy problem of \eqref{HNLS on T3} is locally well-posed in $H^{s_c}(\mathbb T^3)$. 
     For $k=   1$, the Cauchy problem of \eqref{HNLS on T3} is locally well-posed in $H^{s  }(\mathbb T^3)$ for any $s>s_c=1/2$.
 \end{thm}
 \begin{thm}\label{ill-posedness of HNLS on T^3}
     For $k=1$ and $T>0$ be arbitrarily small. Assume the data-to-solution map $u_0 \mapsto u(\cdot)$ associated with \eqref{HNLS on T3} on smooth data extends continuously to a map from $H^{1/2}(\mathbb T^3)$ to $C([0,T];H^{1/2}(\mathbb T^3))$. Then this map will not be $C^3$ at the origin.
 \end{thm}

 The outline of this paper is as follows. In Section~\ref{Section: Strichartz Estimate} we prove Theorem~\ref{sharp T3 L4 hyperbolic Strichartz}. 
 We take the Fourier transform and reduce the $L^4$ estimate \eqref{est:L4}  to  a counting argument for parallelograms with vertices in given sets.  We distinguish two cases depending on whether   the sides of the parallelograms lie on a cone.  

 In Section~\ref{Section: Local Well-posedness} we prove Theorem~\ref{local well-posedness of HNLS on T^3}  based on a multilinear estimate and contraction mapping argument. Then we construct specific solutions to prove Theorem~\ref{ill-posedness of HNLS on T^3}.

\section{Strichartz Estimate}\label{Section: Strichartz Estimate}
\subsection{Notations}
 We denote $A\lesssim B$ or $A=O(B)$ if $A\leq CB$ holds for some constant $C>0$ independent with $A$ and $B$. We write $A\approx B$ if $A\lesssim B$ and $B\lesssim A$.
 We denote $\#S$ the cardinality of finite set $S$. For integers $a,b$, we denote $a|b$ if $a^{-1}b\in\mathbb Z$.
 For $f \in L^2(\mathbb T^3)$, the Fourier coefficients of $f$ are given by
\[ \hat f(k) = \int_{\mathbb T^3} f(x) \me^{-2\pi\mi k\cdot x} \dif x, \quad k\in\mathbb{Z}^3,\]
and the Fourier series of $f$ is
\[ f(x) =  \sum_{k\in\mathbb Z^3} \hat f(k) \me^{2\pi\mi k\cdot x}.\]
The series converges in $L^2(\mathbb T^3)$ sense.
For any subset $S\subset \mathbb Z^3$, we denote $P_S$ for the Fourier multiplier with symbol $\chi_S$, i.e.
\[ P_Sf = \sum_{k\in S} \hat f(k) \me^{ 2\pi\mi k\cdot x}. \]
 In this paper, $N$ will always be a dyadic integer, i.e.  $N=2^n$ for some $n\in\mathbb N$.
For $S = {\{ k \in \mathbb Z^3 \mid  N \leq |k| < 2N \}}$ we simply write $P_S$ as $P_N$, and
\[ P_{\leq N} f= \sum_{M\leq N,\ M \text{ dyadic}} P_Mf, \quad P_{>N}f = f-P_{\leq N}f. \]
For $s\in\mathbb R$, the Sobolev space $H^s(\mathbb T^3)$ is the set of all functions $f\in L^2(\mathbb T^3)$ such that the norm
\[ \| f \|_{H^s(\mathbb{T}^3) }: = \left( \sum_{k\in\mathbb{Z}^3} \left<k\right>^{2s} |\hat f(k)|^2 \right)^{1/2} \]
is finite, where $\left< k \right> = \sqrt{1+|k|^2}$.

\subsection{Facts from incidence geometry}We need the Szeme\'{r}edi-Trotter theorem from incidence geometry. An incidence is defined as a point-curve pair so that the point lies on the curve.  The problem is to bound the number of incidence that are possible for certain classes of curves.
 \begin{thm}[\cite{Szemeredi-Trotter,Pach-Sharir2004}Points-lines incidences]\label{Szemeredi-Trotter theorem}
     Let $\mathcal{P}$ be a set of $n$ points and $\mathcal{L}$ be a set of $m$ lines. Then the number of incidences between $\mathcal{P}$ and $\mathcal{L}$ is  $O(n^{2/3}m^{2/3}+m+n)$.
 \end{thm}
 \begin{coro}[\cite{Pach-Sharir2004}]\label{cor:ST}
     Let $\mathcal{P}$ be a set of $n$ points and $\mathcal{L}$ be a set of lines. Suppose that every line in $\mathcal L$ contains at least $k\geq2$ points of $\mathcal{P}$. Then the number of incidences between $\mathcal{P}$ and $\mathcal{L}$ is  $O(n^{2}/k^2 + n)$.
 \end{coro}
 We also need the following upper bound on points-circles incidences.
 \begin{thm}[Points-circles incidences on sphere]\label{thm:pointcircleincidence}
     Let $\mathcal{P}$ be a set of $n$ points on the unit sphere and $\mathcal{C}$ be a set of $m$ great circles on the unit sphere. Then the number of incidences between $\mathcal{P}$ and $\mathcal{C}$ is  $O(n^{2/3}m^{2/3}+m+n)$.
 \end{thm}
 \begin{proof}
                It suffices to consider the incidences on a half sphere, since $\mathbb{S}^2$ can be covered by eight half spheres. We define the map $\Psi\colon \{(x_1,x_2,x_3)\in \mathbb{S}^2 \mid x_3>0\} \to \mathbb R^2$, $\Psi(x_1,x_2,x_3)=(x_1/x_3,x_2/x_3)$. It's easy to see that $\Psi$ is a bijection, hence it preserves the number of incidences. Besides, $\Psi$ maps the intersection of great circles and half sphere into lines on the plane, the conclusion follows from Theorem~\ref{Szemeredi-Trotter theorem}.   
 \end{proof}
\begin{remark}
    The same points-circles incidences estimate on the sphere holds true if no three circles intersect in two common points; for example, if all circles are congruent and are not great circles on the sphere, see \cite[Section 5.3]{MR1032370} for more information. 
\end{remark}
 
 \subsection{Preparation}
 We will focus on the proof of the $L^4$ Strichartz estimate \eqref{est:L4}.
 We denote $A$ the diagonal matrix $\operatorname{diag}\{1,-1,-1\}$, and $h(\xi) = \xi \cdot A\xi$ denotes the inner product of $\xi\in \mathbb Z^3$ and $A\xi$. With these notations, we may write
 \begin{equation}
     \me^{\mi t\Box}\phi(x) = \sum_{\xi\in\mathbb Z^3}\hat{\phi}(\xi) \me^{2\pi\mi (x\cdot \xi + th(\xi))}.
 \end{equation}
 As a result, its $L^4$ norm is given by
 \begin{align}
    & \int_{[0,1]\times \mathbb T^3} | \me^{\mi t\Box}\phi(x)|^4 \dif t\dif x \\=&{} \sum_{\xi_1,\xi_2,\xi_3,\xi_4\in\mathbb{Z}^3} \overline{\hat\phi(\xi_1)}\hat\phi(\xi_2)  \overline{\hat{\phi}(\xi_3)} \hat\phi(\xi_4) \int_{[0,1]\times \mathbb T^3} \me^{2\pi\mi (x\cdot\sum_{i=1}^4(-1)^i\xi_i + t\sum_{i=1}^4 (-1)^ih(\xi_i) )} \dif t\dif x \\ =& \sum_{(\xi_1,\xi_2,\xi_3,\xi_4)\in\mathcal{Q}} \overline{\hat\phi(\xi_1)}\hat\phi(\xi_2)  \overline{\hat{\phi}(\xi_3)} \hat\phi(\xi_4), \label{formula:L4}
 \end{align}
 where
 \begin{align}
     \mathcal{Q} &= \Big\{ (\xi_1,\xi_2,\xi_3,\xi_4) \in \mathbb{Z}^{3\times4} \Bigm| \sum_{i=1}^4(-1)^i\xi_i=0,\ \sum_{i=1}^4(-1)^ih(\xi_i)=0 \Big\}
     \\& = \Big\{ (\xi_1,\xi_2,\xi_3,\xi_4) \in \mathbb{Z}^{3\times4} \Bigm| \sum_{i=1}^4(-1)^i\xi_i=0,\ {(\xi_1-\xi_2)\cdot A(\xi_1-\xi_4)=0} \Big\}.
 \end{align}
 The first condition indicates that $\xi_1,\xi_2,\xi_3,\xi_4$ form a parallelogram, while the second condition indicates some relations between the directions of the sides. We denote
 \begin{equation}
     \cone = \{ \xi \in \mathbb Z^3 \mid \xi \cdot A\xi =0 \},
 \end{equation}
 which will play a role in our arguments.
 We denote $\mathcal{H}(S)$ the set of all planes (not necessarily passing through the origin) with normal vector belonging to $S\subset \mathbb Z^3$.
 The set $\mathcal{Q}$ can be decomposed as  $ \mathcal{Q}_1 \cup \mathcal{Q}_2$, where
 \begin{equation}\label{notation:Q1}
     \mathcal{Q}_1 = \{ (\xi_1,\xi_2,\xi_3,\xi_4) \in \mathcal{Q} \mid \xi_1 - \xi_2 \notin \cone, \text{ and } \xi_1 - \xi_4 \notin \cone \},
 \end{equation}
 \begin{equation}  \label{notation:Q2}
     \mathcal{Q}_2 = \{ (\xi_1,\xi_2,\xi_3,\xi_4) \in \mathcal Q \mid \xi_1 - \xi_2 \in \cone,\ \text{ or  }\ \xi_1 - \xi_4 \in \cone \}.
 \end{equation}
 We also denote the four-linear operators
 \begin{equation}
     \Omega_1(f_1,f_2,f_3,f_4) = \sum_{{(\xi_1,\xi_2,\xi_3,\xi_4)} \in \mathcal{Q}_1 } f_1(\xi_1)f_2(\xi_2)f_3(\xi_3)f_4(\xi_4),
 \end{equation}
 \begin{equation}
     \Omega_2 {(f_1,f_2,f_3,f_4)} = \sum_{{(\xi_1,\xi_2,\xi_3,\xi_4)} \in \mathcal{Q}_2} f_1(\xi_1)f_2(\xi_2)f_3(\xi_3)f_4(\xi_4).
 \end{equation}
 For simplicity, we write $\Omega_1(f)$ and $\Omega_2(f)$ instead of $\Omega_1(f,f,f,f)$ and $\Omega_2(f,f,f,f)$.

We introduce more notations. For any    $M>0$,
\begin{equation}
    \cone_M = \{  \xi \in \cone  \setminus\{0\} \mid  |\xi| /\gcd(\xi) \leq M   \},
\end{equation}
\begin{equation} 
    \cone^{\mathrm{irr}}_M = \{ \xi\in\cone_M \mid \gcd(\xi)=1 \},
\end{equation}
where $\gcd(\xi)$ denotes the greatest common divisor of coordinates of $\xi\in\mathbb Z^3$.
\begin{lem}\label{lem:sizeCone-irr}
    We have the size estimate $\#\cone^{\mathrm{irr}}_M \lesssim M$.
\end{lem}
\begin{proof}
   
    Suppose $(x_1,x_2,x_3) \in \cone_M^{\text{irr}}$, i.e. $x_1^2=x_2^2+x_3^2$ and $\gcd(x_1,x_2,x_3)=1$. It's clear that $x_2,x_3$ cannot be both even, we may assume $x_3$ is odd,  and $x_3 = \pm  p_1^{\alpha_1}\dots p_r^{\alpha_r}$ is the prime factorization. We note that
\begin{equation}
    p_i^{2\alpha_i} | x_3^2 = (x_1-x_2)(x_1+x_2),
\end{equation}
so there exists some $\gamma_i\in\mathbb N$ such that $p^{\gamma_i}_i | (x_1-x_2)$ and $p_i^{2\alpha_i-\gamma_i} | (x_1+x_2)$.
If $\gamma_i \neq 0,2\alpha_i$, then $p_i$ divides both $x_1-x_2$ and $x_1+x_2$ and hence $p_i$ divides both $2x_1$ and $2x_2$. Consequently $p_i|\gcd(x_1,x_2,x_3)$, which is a contradiction. Hence we have exactly one of $p_i^{2\alpha_i}|(x_1-x_2)$ and $p_i^{2\alpha_i}|(x_1+x_2)$ holds.  Denote $I = \{ 1\leq i\leq r\mid p_i^{2\alpha_i} \text{ divides } x_1-x_2 \}$ and 
\begin{equation}
    m = \prod_{i\in I} p_i^{\alpha_i},\quad n = \prod_{i\notin I} p_i^{\alpha_i} = |x_3|/m,
\end{equation}
where the product is defined to be $1$ if the index set is empty.
Then $\gcd(m,n)=1$ and we have $(x_1-x_2 ,x_1+x_2 )= \pm (m^2,n^2)$, or equivalently
\begin{equation}
    (x_1, x_2) = \pm \left( \frac{n^2+m^2}{2}, \frac{n^2-m^2}2 \right).
\end{equation}
Therefore, each point $(x_1,x_2,x_3) \in \cone_M^{\text{irr}}$ can be represented by a pair $(m,n) \in \mathbb Z^2$ satisfying $m^2+n^2 \lesssim M$, and hence $\#\cone_M^{\text{irr}} \lesssim M$.
\end{proof}

 Before the start of proofs, we briefly talk  about the geometry of parallelograms in $\mathcal{Q}_2$.
 Due to symmetry, we may only consider the case $\xi_1-\xi_2 \in \cone$. For each ${(\xi_1,\xi_2,\xi_3,\xi_4)} \in \mathcal{Q}_2$ such that the parallelogram is non-degenerate, the four vertices are contained in some plane $H$. From the definitions we can see that $A(\xi_1-\xi_2)$ is perpendicular to both $\xi_1-\xi_2$ and $\xi_1-\xi_4$. Hence $A(\xi_1-\xi_2)$ is a normal vector of $H$, and it belongs to $\cone$. When the the parallelogram is degenerate, we can still find a plane $H$ containing all vertices and its normal vector belongs to $\cone$.

 On the other hand, let $H$ be a plane with normal vector $n$ which belongs to $\cone$, and suppose $H$ contains the four vertices of ${(\xi_1,\xi_2,\xi_3,\xi_4)} \in \mathcal{Q}$. Clearly, $n$ is perpendicular to  $\xi_1-\xi_2$, $\xi_1-\xi_4$ and also $An$. Notice that
 \begin{equation}
     0 =  (\xi_1-\xi_2) \cdot n = A(\xi_1-\xi_2) \cdot An,
 \end{equation}
 and ${(\xi_1,\xi_2,\xi_3,\xi_4)} \in \mathcal{Q}$ indicates that $A(\xi_1-\xi_2) \cdot (\xi_1-\xi_4) = 0 $. Hence we know that
 \begin{equation}
    \operatorname{span}_{\mathbb R}\{ n,A(\xi_1-\xi_2) \} \perp \operatorname{span}_{\mathbb R} \{ \xi_1-\xi_4,An \}.
 \end{equation}
But the sum of their dimensions is no more than $3$.
Therefore, we have that either $\xi_1-\xi_4$ is a multiple of $An$ or $A(\xi_1-\xi_2)$ is a multiple of $n$, and in both cases ${(\xi_1,\xi_2,\xi_3,\xi_4)}$ must belong to $\mathcal{Q}_2$.
 As a result, we may write 
 \begin{equation}\label{decomposition of Q_2 as union of H^4 cap Q}
 \mathcal{Q}_2 =
     \bigcup_{H \in \mathcal{H}(\cone)} {\{ {(\xi_1,\xi_2,\xi_3,\xi_4)} \in \mathbb{Z}^{3\times 4} \mid \xi_i\in H,\ 1\leq i\leq 4 \} } \cap \mathcal{Q}.
 \end{equation}

\subsection{The contributions of parallelograms with side on the cone}
 \begin{prop}\label{number of parallelograms with side on cone}
 For $f\colon \mathbb Z^3 \to \mathbb R_+$ supported on a finite subset $S\subset \mathbb Z^3$, we have
\begin{equation}
     \Omega_2(f)\lesssim \diam(S) \| f \|_{\ell^2(\mathbb Z^3)}^4.
\end{equation}
\end{prop}
\begin{proof}

    It suffices to consider the case $\xi_1-\xi_2 \in \cone$ and $\xi_1\neq \xi_4$. 
For given $(\xi_1,\xi_4)$,
   
    from previous discussion we know there exists some plane $H$ contains both $\xi_1,\xi_4$ and its normal vector belongs to $\cone$. It's not hard to check that there exist at most two such planes. For each such plane $H$, $A(\xi_1-\xi_2)$ is a multiple of its normal vector and hence $\xi_2$ lies on a line $\ell$ passing through $\xi_1$ with direction determined by $H$.
    
    We may write  $\xi_2 \in \ell \cap S \subset \mathbb Z^3$ as $\xi_1+r\xi$ with $\xi \in \mathbb Z^3\setminus\{0\}$, thus $r\xi \in \mathbb Z^3$ and $r$ belongs to an interval of length $|\xi|^{-1} \diam(S)$.  From B\'{e}zout's identity, we know  $\gcd(\xi)$ can be written as linear combination of coordinates of $\xi$ with integer coefficients, we have $r \gcd(\xi) \in \mathbb Z$. Thus
    \begin{equation}\label{number of points on intersection of line and S}
        \#(\ell \cap S) = \# \Bigl\{ r\in \frac{1}{\gcd(\xi)} \mathbb{Z} \Bigm| \xi_1+r\xi \in \ell\cap S \Bigr\} \leq \frac{\gcd(\xi)}{|\xi|} \diam(S),
    \end{equation}
    which implies for each pair $(\xi_1,\xi_4)$, there exists at most $O(\diam(S))$ many choices of $(\xi_2,\xi_3)$ such that ${(\xi_1,\xi_2,\xi_3,\xi_4)} \in \mathcal{Q}_2$. We denote all the possible choices as ${(\xi_2,\xi_3)\in \mathcal{R}(\xi_1,\xi_4)}$. 
    
    On the other hand, for given $(\xi_2,\xi_3)$, we can also apply the same argument to $(\xi_1,\xi_4)$, 
    and hence for each pair $(\xi_2,\xi_3)$, there exists at most $O(\diam(S))$ many choices of $(\xi_1,\xi_4)$ such that ${(\xi_1,\xi_2,\xi_3,\xi_4)} \in \mathcal{Q}_2$. As a result,
    \begin{align}
        \Omega_2(f) &= \sum_{\xi_1,\xi_4} \Big( f(\xi_1)f(\xi_4) \sum_{(\xi_2,\xi_3)\in \mathcal{R}(\xi_1,\xi_4)} f(\xi_2)f(\xi_3)  \Big) \\&
        \leq \Big( \sum_{\xi_1,\xi_4} |f(\xi_1)f(\xi_4)|^2 \Big)^{1/2} \Big( \sum_{\xi_1,\xi_4} \Big(  \sum_{(\xi_2,\xi_3)\in \mathcal{R}(\xi_1,\xi_4)} f(\xi_2)f(\xi_3)  \Big)^2 \Big)^{1/2} \\&
        \lesssim \diam(S)^{1/2} \|f\|^2_{\ell^2(\mathbb Z^3)} \Big( \sum_{\xi_1,\xi_4}   \sum_{(\xi_2,\xi_3)\in \mathcal{R}(\xi_1,\xi_4)} | f(\xi_2)f(\xi_3)  |^2 \Big)^{1/2} \\&
        =\diam(S)^{1/2} \|f\|^2_{\ell^2(\mathbb Z^3)} \Big( \sum_{\xi_2,\xi_3}   \sum_{(\xi_1,\xi_4)\in\mathcal{R}(\xi_2,\xi_3)} | f(\xi_2)f(\xi_3)  |^2 \Big)^{1/2} \\&
        \lesssim \diam(S) \|f\|_{\ell^2(\mathbb Z^3)}^4.
    \end{align}
    
\end{proof}

\begin{prop}\label{characterization of bad sets}
   
    For $f\colon \mathbb Z^3 \to \mathbb R_+$ supported on a finite subset $S\subset \mathbb Z^3$ and $M>0$, 
     there exists at most $O(M^3)$  planes $\{H_i\} \subset \mathcal{H}(\cone^{\mathrm{irr}}_M)$, such that 
    \begin{equation}
        \| f\chi_H \|_{\ell^2(\mathbb Z^3)}^2 \geq M^{-2} \|f\|_{\ell^2(\mathbb Z^3)}^2.
    \end{equation}
   If we denote    $f^{\mathrm{error}}:=f\chi_{S \setminus \cup_i H_i}$, then we have 
    \begin{equation}\label{estimate for f^error}
        \Omega_2(f^{\mathrm{error}}) \lesssim M^{-1}\diam(S) \| f\|_{\ell^2(\mathbb Z^3)}^4.
    \end{equation}
\end{prop}
\begin{proof}
  
   We set $\{H_i\}$ to be the set of all planes $H$ with normal vector in $\cone_M^{\mathrm{irr}}$ and satisfying $\| f\chi_H \|_{\ell^2(\mathbb Z^3)}^2 \geq M^{-2} \|f\|_{\ell^2(\mathbb Z^3)}^2$.
   For each $n\in \cone_M^{\mathrm{irr}}$, the planes with normal vector $n$  are parallel with each other, which implies
   \begin{equation}
       \# \{ H \in \mathcal{H}(\{n\}) \mid \| f\chi_H \|_{\ell^2(\mathbb Z^3)}^2 \geq M^{-2} \|f\|_{\ell^2(\mathbb Z^3)}^2 \} \leq M^2,
   \end{equation}
   thus $\#\{H_i\} \leq M^2 \# \cone_M^{\mathrm{irr}} \lesssim M^3$.

 It remains to verify \eqref{estimate for f^error}.
   By the decomposition \eqref{decomposition of Q_2 as union of H^4 cap Q}   and the facts that $0\leq f^{\mathrm{error}} \leq f$, $f^{\mathrm{error}} \chi_{H_i}=0$ and $\mathcal{H}(\cone_M^{\mathrm{irr}})= \mathcal{H}(\cone_M)$, we get
   \begin{align}
       \Omega_2(f^{\mathrm{error}}) &\leq \sum_{ H \in \mathcal{H}(\cone) } \Omega_2(f^{\mathrm{error}} \chi_H) \\& \leq \sum_{ H \in \mathcal{H}(\cone \setminus \cone_M) } \Omega_2(f  \chi_H) + \sum_{ H \in \mathcal{H}(\cone_M^{\mathrm{irr}} ) \setminus \{H_i\} } \Omega_2(f \chi_H).
   \end{align} 
    Recall the estimate \eqref{number of points on intersection of line and S}, by using arguments similar to that in proof of Proposition~\ref{number of parallelograms with side on cone}, we have
    \begin{equation}
       \sum_{ H \in \mathcal{H}(\cone \setminus \cone_M) } \Omega_2(f  \chi_H)
       \lesssim M^{-1}\diam(S) \| f\|_{\ell^2(\mathbb Z^3)}^4.
    \end{equation}
On the other hand, for each $H \in \mathcal{H}(\cone_M^{\mathrm{irr}}) \setminus \{H_i\},$
we have $\| f\chi_H \|_{\ell^2(\mathbb Z^3)}^2 \leq M^{-2} \| f \|_{\ell^2(\mathbb Z^3)}^2$, and
\begin{align}
   \sum_{ H \in \mathcal{H}(\cone_M^{\mathrm{irr}} ) \setminus \{H_i\} } \Omega_2(f \chi_H) &
   ={} \sum_{n\in \cone_M^{\mathrm{irr}}} \sum_{\substack{H \notin \{H_i\} \\ H \perp n}} \Omega_2(f\chi_H)
    \\&
    \lesssim{} \sum_{n\in \cone_M^{\mathrm{irr}}}  \sum_{\substack{H \notin \{H_i\} \\ H \perp n}}\diam(S) \| f\chi_H \|_{\ell^2(\mathbb Z^3)}^4 \\&
    \leq{} \frac{\diam(S)  \|f \|_{\ell^2(\mathbb Z^3)}^2 }{M^2}\sum_{n\in \cone_M^{\mathrm{irr}}}  \sum_{\substack{H \notin \{H_i\} \\ H \perp n}} \|f\chi_H\|_{\ell^2(\mathbb Z^3)}^2 \\&
    \leq{} M^{-2}\#\cone_M^{\mathrm{irr}}  \diam(S) \|f\|_{\ell^2(\mathbb Z^3)}^4 \\& \lesssim M^{-1} \diam(S) \| f ||_{\ell^2(\mathbb Z^3)}^4 .
\end{align}
Here we used Proposition~\ref{number of parallelograms with side on cone} for the first inequality and Lemma~\ref{lem:sizeCone-irr} for the last inequality.
    Hence
    \begin{equation}
        \Omega_2(f^{\mathrm{error}}) \lesssim M^{-1}\diam(S) \| f\|_{\ell^2(\mathbb Z^3)}^4.
    \end{equation}
\end{proof}
 Combining the above two propositions, we see that if $\Omega_2(f)$ is large, then $f$ should concentrate on few planes. Thus we get more information about the geometric structure of the distribution of $f$. This observation  is crucial in our proof.
\subsection{The contributions of parallelograms without side on the cone}

\begin{prop}\label{number of parallelograms without side on cone}
    For $f=\chi_S$ with $S$  a finite subset of $\mathbb{Z}^3$, we have
    \begin{equation}
        \Omega_1(f) \lesssim (\#S)^{7/3}=(\#S)^{1/3} \| f\|^4_{\ell^2(\mathbb Z^3)}.\label{bound:Omega1}
    \end{equation}
\end{prop}
This can be proved by the same method in \cite{Apfelbaum-Sharir2005}. For $\xi\in\mathbb{Z}^3$ and $\ell,\ell'$ two lines in $\mathbb R^3$, we denote
\begin{equation}
    \crf(\xi,\ell,\ell') = \begin{cases}
        1,& \text{if } \xi = \ell \cap \ell' \text{ and }v_\ell \cdot Av_{\ell'}=0,\\0,&\text{otherwise,}
    \end{cases}
\end{equation}
where $v_\ell$ denotes the direction vector of the line $\ell$.
We need the following lemma.
\begin{lem}\label{crossing-incidence estimate with fixed point}
    For fixed $\xi \in \mathbb{Z}^3$, let  $\mathcal{L},\mathcal{L'}$ be two finite families of lines passing through $\xi$. Then
    \begin{equation}
        \sum_{\ell \in \mathcal{L}} \sum_{\ell'\in\mathcal{L}'} \crf(\xi,\ell,\ell') \lesssim (\# \mathcal{L})^{2/3}(\#\mathcal{L}')^{2/3} + \#\mathcal{L}+\#\mathcal{L}'.
    \end{equation}
\end{lem}
\begin{proof}
For each line $\ell$, we denote its direction vector as $v_\ell \in \mathbb{S}^2$, and $c_\ell = \{ v\in \mathbb{S}^2 \mid v \cdot Av_{\ell} =0 \}$. Then $\crf(\xi,\ell,\ell') = 1$ is equivalent to $v_\ell \in c_{\ell'}$. Set $\mathcal{P} = \{ v_\ell   \mid \ell \in \mathcal{L} \}$    and $\mathcal{C} = \{ c_{\ell'} \mid   \ell'\in \mathcal{L}'  \}$, then the summation of $\crf(\xi,\ell,\ell')$ is bounded by the number of incidences between points in $\mathcal{P}$ and circles in $\mathcal{C}$, see Theorem~\ref{thm:pointcircleincidence}.
\end{proof}

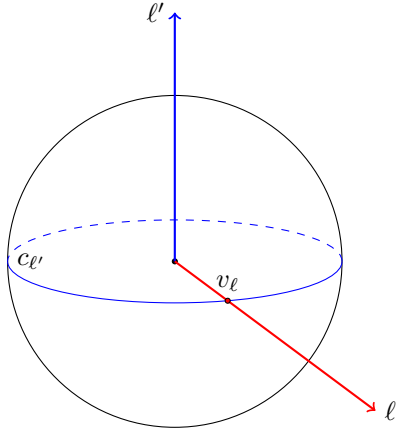
\begin{figure}[htbp]
\begin{tikzpicture}[scale=1.1]
\draw[fill] (0,0) circle (0.03);
\draw (0,0) circle (2);
\draw[blue] (-2,0) arc (180:360:2 and 0.5);
\draw[dashed,blue] (2,0) arc (0:180:2 and 0.5);
\draw[->,thick,blue] (0,0) -- (0,3);
\draw[->,thick,red] (0,0) -- (12/5,-9/5);
\draw[fill=red] (0.6325,-0.4743) circle (0.03);
\node[above] at (0.6325,-0.4743) {$v_\ell$};
\node[right] at (12/5,-9/5) {$\ell$};
\node[left] at (0,3) {$\ell'$};
\node[right] at (-2,0) {$c_{\ell'}$};
\end{tikzpicture}
\centering
\end{figure}

\begin{proof}[Proof of Proposition~\ref{number of parallelograms without side on cone}]
    For each dyadic integer $s$, we put $\mathcal{L}_s$ to be the set of lines $\ell$ such that $s\leq \#(\ell\cap S) < 2s$ and the direction vector of $\ell$ doesn't belong to $\cone$. Then
    \begin{align}
        \Omega_1(f) &= \sum_{s,t\text{ dyadic}} \#\{ {(\xi_1,\xi_2,\xi_3,\xi_4)}\in \mathcal{Q}_1 \mid \ell_{\xi_1\xi_2} \in \mathcal{L}_s,\ell_{\xi_1,\xi_4}\in\mathcal{L}_t \} \\&
        \leq \sum_{s,t\text{ dyadic}} \sum_{\ell \in \mathcal{L}_s} \sum_{\ell'\in \mathcal{L}_t} \sum_{\xi_1\in S} \crf(\xi_1,\ell,\ell')\#(\ell\cap S)\#(\ell'\cap S). \label{bound:Omega1-inci}
    \end{align}
  For $s,t\leq (\#S)^{1/3}$, we use Lemma~\ref{crossing-incidence estimate with fixed point} to estimate
    \begin{align}
        \sum_{s,t\text{ dyadic}}  \sum_{\xi_1\in S}st \Big( \sum_{\ell \in \mathcal{L}_s} \sum_{\ell'\in \mathcal{L}_t} \crf(\xi_1,\ell,\ell') \Big) \lesssim \sum_{s,t\text{ dyadic}} st \sum_{\xi_1\in S}( (\mathcal{J}_{\xi_1}^s\mathcal{J}_{\xi_1}^t)^{2/3} + \mathcal{J}_{\xi_1}^s+\mathcal{J}_{\xi_1}^t ),
    \end{align}
    where $\mathcal{J}_{\xi_1}^s,\mathcal{J}_{\xi_1}^t$ denote the number of lines in $\mathcal{L}_s,\mathcal{L}_t$ passing through $\xi_1$. By Corollary~\ref{cor:ST}, we have
    \[
    \sum_{\xi_1\in S} \mathcal{J}_{\xi_1}^s \lesssim \frac{(\#S)^2}{s^2},\quad \sum_{\xi_1\in S} \mathcal{J}_{\xi_1}^t \lesssim \frac{(\#S)^2}{t^2},
    \]
    hence
    \begin{equation}
        \sum_{s,t \leq (\#S)^{1/3}\text{ dyadic}} st \sum_{\xi_1\in S}(\mathcal{J}_{\xi_1}^s+\mathcal{J}_{\xi_1}^t) \lesssim (\#S)^{7/3}.
    \end{equation}
    On the other hand, since all the lines passing through $\xi_1$ are pairwise disjoint (excluding the common point $\xi_1$), we have $\mathcal{J}_{\xi_1}^s \lesssim \#S/s$ and $\mathcal{J}_{\xi_1}^t \lesssim \#S/t$. Therefore,
    \begin{align}
        \sum_{s,t\text{ dyadic}} st \sum_{\xi_1\in S} (\mathcal{J}_{\xi_1}^s\mathcal{J}_{\xi_1}^t)^{2/3} &\lesssim \sum_{s,t\text{ dyadic}} st \sum_{\xi_1 \in S} \left( \frac{(\#S)^2}{st} \right)^{1/6} (\mathcal{J}_{\xi_1}^s\mathcal{J}_{\xi_1}^t)^{1/2} \\&\leq \sum_{s,t\text{ dyadic}} {(\#S)^{1/3}}{(st)^{5/6}} \Big( \sum_{\xi_1\in S} \mathcal{J}_{\xi_1}^s \Big)^{1/2}\Big( \sum_{\xi_1\in S}  \mathcal{J}_{\xi_1}^t\Big)^{1/2} \\&
        \lesssim \sum_{s,t\text{ dyadic}} {(\#S)^{1/3}}{(st)^{5/6}}\frac{\#S}{s}\frac{\#S}{t} \lesssim (\#S)^{7/3}.
    \end{align}
    This proves \eqref{bound:Omega1} when  $s,t\leq (\#S)^{1/3}$.
    
    For $s\geq(\#S)^{1/3}$ or $t\geq (\#S)^{1/3}$, we assume $s\geq (\#S)^{1/3}$ without loss of generality. We notice that for any fixed $\ell$ and $\xi_1\in\ell$, the lines $\ell'$ such that $\crf(\xi_1,\ell,\ell')\neq 0$ belong to some plane determined by $\ell$ and $\xi_1$. Furthermore, for different choices of $\xi_1$, these planes are pairwise disjoint due  to the fact that the direction vector of $\ell$ doesn't belong to $\cone$. Thus all these lines $\ell'$ are pairwise disjoint (excluding the possible common points on $\ell$), and
    \begin{equation}
       \sum_t \sum_{\ell'\in\mathcal{L}_t} \sum_{\xi_1\in S}  \crf(\xi_1,\ell,\ell') \#(\ell'\cap S) \lesssim \#S. \label{est:case2-1}
    \end{equation}
    On the other hand,
    by Corollary~\ref{cor:ST} we have
    \begin{equation}
        \sum_{s\geq (\#S)^{1/3}} \sum_{\ell \in \mathcal{L}_s} \#(\ell \cap S) \lesssim (\#S)^{4/3}.\label{est:case2-2}
    \end{equation}
\eqref{bound:Omega1-inci}, \eqref{est:case2-1} and \eqref{est:case2-2} together imply the bound \eqref{bound:Omega1}.
\end{proof}
\subsection{Proof  of Theorem~\ref{sharp T3 L4 hyperbolic Strichartz}}
Proposition~\ref{number of parallelograms without side on cone} deals with characteristic functions. To extend this result to the general case, we employ an atomic decomposition that reduces an arbitrary function to a sum of characteristic functions.
This leads us to the task of estimating the multilinear form $\Omega_1(f_1, f_2, f_3, f_4)$ with $\operatorname{supp} f_i\subset S_i$.  A similar problem was tackled by Herr-Kwak~\cite{Herr-Kwak}, who performed a very careful analysis to bound the number of parallelograms in terms of the size of the sets $S_i$. However, the problem becomes more intricate in 3 dimensions. To overcome this, we apply Proposition~\ref{characterization of bad sets} to further reduce the problem to the case where the support of each function $f_i$ lies in a plane $H_i\in \mathcal{H}(\cone^{\mathrm{irr}}_M)$. The particular geometric structure of these planes becomes crucial for obtaining the desired estimate.  Now we turn to the details. 

\begin{proof}[Proof of Theorem~\ref{sharp T3 L4 hyperbolic Strichartz}]

We set $f = |\hat\phi| \chi_{[-N,N]^3}$  and enumerate $\mathbb{Z}^3 \cap [-N,N]^3$ as $\xi_1,\xi_2,\dots$ such that
\begin{equation}
    f(\xi_1) \geq f(\xi_2) \geq \dots.
\end{equation}
Let $S_j = \{ \xi_{2^{j}},\dots,\xi_{2^{j+1}-1} \}$, $f_j = f(\xi_{2^j}) \chi_{S_j}$ and $\lambda_j = 2^{j/2} f(\xi_{2^j})$ for $0\leq j \leq {j_{\text{max}}}$ with $ 2^{j_{\text{max}}}\lesssim N^3$. We have 
    $\# S_j \leq 2^j$, $f \leq \sum_{j} f_j$, $|\lambda_j|=\|f_j\|_{\ell^2(\mathbb{Z}^3)}$  and
    \begin{equation}
        \| \lambda_j\|_{\ell^2_{j\leq {j_{\text{max}}}}}^2 = \sum_{j=0}^{{j_{\text{max}}}} 2^j|f(\xi_{2^j})|^2 \leq |f(\xi_1)|^2 + \sum_j \| f\chi_{S_{j-1} }\|_{\ell^2(\mathbb Z^3)}^2 \lesssim \|f\|_{\ell^2(\mathbb Z^3)}^2. \label{est:lambda-f}
    \end{equation}
Let $\delta>0$ sufficiently small.
Given such a decomposition of $f$, we say $j$ is good if $\Omega_2(f_j) \lesssim N^{1-\delta} \| f_j \|_{\ell^2(\mathbb Z^3)}^4$, otherwise we say $j$ is bad. 
For each bad $j$, by taking $M=N^\delta$  in Proposition~\ref{characterization of bad sets}, we can find at most $O(N^{3\delta})$
planes $\{H_{j}^{i_j}\} \subset \mathcal{H}(\cone_M^{\mathrm{irr}})$, such that
$f_j^{\mathrm{error}}:=f\chi_{S_j \setminus \cup_{i_j} H_{j}^{i_j}}$ satisfies
    \begin{equation}
        \Omega_2(f_j^{\mathrm{error}}) \lesssim N^{1-\delta} \| f_j\|_{\ell^2(\mathbb Z^3)}^4. \label{def:fgood}
    \end{equation}
We denote $f_j^{\mathrm{good}}=f_j^{}$ if $j$ is good and $f_j^{\mathrm{good}}=f_j^{\mathrm{error}}$ if $j$ is bad, and
\begin{equation}
    f_{\text{bad}} = \sum_{j \text{ bad} } (f_j^{}-f_j^{\mathrm{good}}).
\end{equation}
Then from Proposition~\ref{number of parallelograms without side on cone} and \eqref{def:fgood}, we get bounds for  each $f_j^{\mathrm{good}}$
\begin{align}
     \| \me^{\mi t\Box} \check{f}_j^\text{good} \|^4_{{L_{t,x}^4([0,1]\times \mathbb T^3)}} &= \Omega_1(f_j^\text{good}) + \Omega_2(f_j^\text{good})\\
    & \lesssim  (2^{{j}/{3}}+N^{1-\delta})\|f_j^\text{good}\|^4_{\ell^2(\mathbb{Z}^3)}. \label{est:f-good}
\end{align}
Now we control the contribution of $f_{\text{bad}}$, which can be written as 
\begin{align}
     &\| \me^{\mi t\Box} \check f_{\text{bad}}\|_{{L_{t,x}^4([0,1]\times \mathbb T^3)}}^4 = \Omega_1(f_{\text{bad}}) +\Omega_2(f_{\text{bad}})  .
\end{align} 
Proposition~\ref{number of parallelograms with side on cone} indicates that $\Omega_2(f_{\text{bad}}) \lesssim N \| f_{\text{bad}} \|_{{\ell^2(\mathbb Z^3)}}^4$. Meanwhile, we have the  estimate
\begin{align}
    \Omega_1(f_{\text{bad}}) &\lesssim  \sum_{\substack{j_1,j_2,j_3,j_4 \text{ bad} \\ i_{j_1},i_{j_2},i_{j_3},i_{j_4} \lesssim N^{3\delta}}} \Omega_1(f_{j_1}\chi_{H_{j_1}^{i_{j_1}}},f_{j_2}\chi_{H_{j_2}^{i_{j_2}}}, f_{j_3}\chi_{H_{j_3}^{i_{j_3}}}, f_{j_4}\chi_{H_{j_4}^{i_{j_4}}}) \\
    &\lesssim \sum_{j_1,j_2,j_3,j_4} N^{3/4+O(\delta)} \prod_{k=1}^4 \| f_{j_k} \|_{\ell^2(\mathbb Z^3)}
    \lesssim N \| f_{\text{bad}}\|_{\ell^2(\mathbb Z^3)}^4,
\end{align}
which is a consequence of applying the following Proposition~\ref{prop:Omega1-4f} to each quadruple $(i_{j_1},i_{j_2},i_{j_3},i_{j_4})$ with $M=N^\delta$.

Combining all the estimates, we get
\begin{align}
    \| \me^{\mi t\Box}\check{f} \|_{{L_{t,x}^4([0,1]\times \mathbb T^3)}} &\lesssim  \sum_{j} \| \me^{\mi t\Box} \check{f}_j^\text{good} \|_{{L_{t,x}^4([0,1]\times \mathbb T^3)}} + \| \me^{\mi t\Box} \check f_{\text{bad}}\|_{{L_{t,x}^4([0,1]\times \mathbb T^3)}}\\&
    \lesssim \sum_{j \leq {j_{\text{max}}}} (2^{j/12} + N^{(1-\delta)/4})\lambda_j +  \| \me^{\mi t\Box} \check f_{\text{bad}}\|_{{L_{t,x}^4([0,1]\times \mathbb T^3)}}\\&
    \lesssim N^{1/4} \| \lambda_j \|_{\ell_{j\leq {j_{\text{max}}}}^2} +  \| \me^{\mi t\Box} \check f_{\text{bad}}\|_{{L_{t,x}^4([0,1]\times \mathbb T^3)}}\\
    &\lesssim N^{1/4} \| f\|_{\ell^2(\mathbb Z^3)}.  
\end{align}
From \eqref{formula:L4}, we see $\| \me^{\mi t\Box}\phi \|_{{L_{t,x}^4([0,1]\times \mathbb T^3)}}\leq \| \me^{\mi t\Box}\check{f} \|_{{L_{t,x}^4([0,1]\times \mathbb T^3)}}$. Thus we finish the proof for the $L^4$ estimate \eqref{est:L4} 
, and the general case \eqref{main-est} follows from interpolation of $L^4$ with $L^2$ and $L^\infty$ estimates.

\end{proof}

Now we are left to prove the following:
\begin{prop}\label{prop:Omega1-4f}
    Suppose $n_1,n_2,n_3,n_4 \in \cone^{\mathrm{irr}}_M $ are vectors and $H_1,H_2,H_3,H_4$ are planes such that  $n_j\perp H_j$  for each $j$, and functions $f_j \colon \mathbb Z^3 \to \mathbb R_+$ supported on ${  H_j \cap [-N,N]^3}$. Then 
    \begin{equation}
        \Omega_1{(f_1,f_2,f_3,f_4)}  \lesssim M^{2}N^{1/2}(M^2+N^{1/4}) \prod_{j=1}^4 \| f_j\|_{\ell^2(\mathbb Z^3)}. \label{est:prop-4f}
    \end{equation}
\end{prop}
\begin{proof}
    We decompose $\mathcal{Q}_1$ into
    \begin{equation}
        \bigcup_{(a,b) \in \mathbb{Z}^3\times \mathbb{Z} } \Gamma_{a,b} :=\bigcup \Bigg\{ {(\xi_1,\xi_2,\xi_3,\xi_4)} \in \mathcal{Q}_1 \Biggm| \begin{array}{cc} \xi_1+\xi_3=a=\xi_2+\xi_4 \\ h(\xi_1)+h(\xi_3) = b = h(\xi_2)+h(\xi_4)\end{array}\Bigg\}.
    \end{equation}
Also we denote $\mathcal{A}_{a,b} = \{ \xi \in \mathbb Z^3 \mid h(\xi)+h(a-\xi) = b \}$. Then ${(\xi_1,\xi_2,\xi_3,\xi_4)} \in \Gamma_{a,b}$ implies
$\xi_j \in \mathcal{A}_{a,b}$ for $1\leq j\leq 4$. It's easy to see that
\begin{align}
 \Omega_1{(f_1,f_2,f_3,f_4)} &= \sum_{a,b} \sum_{\Gamma_{a,b}} f_1(\xi_1)f_2(\xi_2)f_3(\xi_3)f_4(\xi_4) \label{est:4f-Gamma} \\&
    = \sum_{a,b} \Big( \sum_{\substack{\xi_1,\xi_3\in \mathcal{A}_{a,b}\\\xi_1+\xi_3=a}} f_1(\xi_1)f_3(\xi_3) \Big)\Big( \sum_{\substack{\xi_2,\xi_4\in \mathcal{A}_{a,b}\\\xi_2+\xi_4=a}} f_2(\xi_2)f_4(\xi_4) \Big) . \\\label{est:4f-Aab}
\end{align}
By Cauchy-Schwarz, it's further bounded by
\begin{equation}
    \bigg( \sum_{a,b} \Big( \sum_{\substack{\xi_1 \in \mathcal{A}_{a,b} }} f_1(\xi_1)f_3(a-\xi_1) \Big)^2 \bigg)^{1/2} \bigg( \sum_{a,b} \Big( \sum_{\substack{\xi_2 \in \mathcal{A}_{a,b} }} f_2(\xi_2)f_4(a-\xi_2) \Big)^2 \bigg)^{1/2}.\label{est:4f-Aab-CS}
\end{equation}

Next we estimate the size of $\mathcal{A}_{a,b}\cap H_1\cap[-N,N]^3$, the same argument holds for $j=2,3,4$. Since $n_1\in \cone_M^{\mathrm{irr}}$ is the normal vector of $H_1$ and $n_1 \perp An_1$, hence $\{n_1,An_1,n_1\times An_1\}$ is an orthogonal basis of $\mathbb R^3$. We decompose $\xi$ with respect to this  orthogonal basis and denote
\begin{equation}
    \xi^\perp = \xi - \frac{  \xi\cdot n_1 }{|n_1|^2} n_1 - \frac{\xi\cdot An_1}{|n_1|^2}An_1 \in \frac{1}{|n_1|^2} \mathbb{Z}^3. \label{definition of perp part}
\end{equation}
Clearly $\xi^\perp$ belongs to the one-dimensional linear subspace spanned by $n_1\times An_1$. For $\xi\in  H_1$, the inner product $\xi\cdot n_1 = c_1$ is constant, and we compute $h(\xi)$ by using the decomposition \eqref{definition of perp part}
\begin{align}
    h(\xi) &= \xi \cdot A\xi = \xi^\perp \cdot A\xi^\perp + (\xi - \xi^\perp)\cdot A(\xi-\xi^\perp) + 2\xi^\perp \cdot A(\xi-\xi^\perp) \\
    &= h(\xi^\perp) + 2c_1 \frac{\xi\cdot An_1}{|n_1|^2} .
\end{align}
As a result, it holds that for $\xi \in \mathcal{A}_{a,b} \cap H_1$
\begin{align}
   b &= h(\xi) + h(a-\xi) = 2h(\xi)+h(a)-2\xi\cdot Aa \\& = 2h(\xi^\perp) + 4c_1 \frac{ \xi \cdot An_1 }{|n_1|^2} + h(a) -2\xi^\perp\cdot Aa - \frac{2( c_1 a\cdot An_1 + (\xi_1\cdot An_1)(a\cdot n_1) )}{|n_1|^2} 
\end{align}
i.e.
\begin{equation}
    \frac{4c_1 - 2a\cdot n_1}{|n_1|^2} \xi\cdot An_1 = -2h(\xi^\perp)+2\xi^\perp\cdot Aa + \Big[ b - h(a) + \frac{2c_1}{|n_1|^2} a\cdot An_1 \Big]. \label{equation of xi,a,b}
\end{equation}

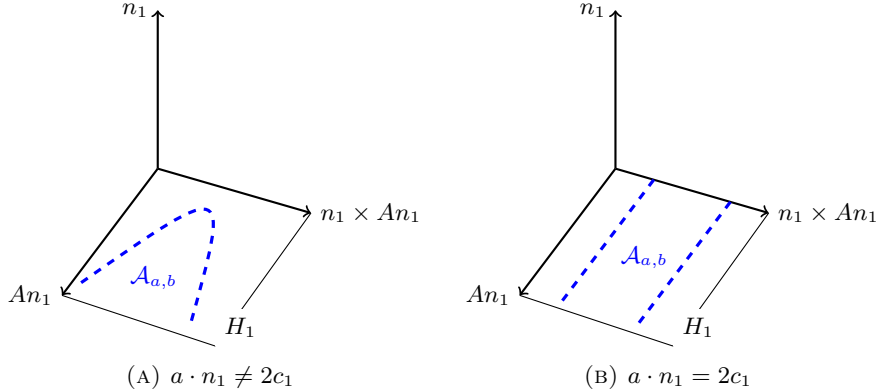
\begin{figure}[htbp]
\centering
\subcaptionbox{$a\cdot n_1\neq 2c_1$}{
\begin{tikzpicture}[scale = 2.1]
    \draw[->,thick] (0,0) -- (0,1);
    \draw[->,thick] (0,0) -- (24/25,-7/25);
    \draw[->,thick] (0,0) -- (-3/5,-4/5);
    \draw (-3/5,-4/5) -- (9/25,-28/25);
    \draw (24/25,-7/25) -- (9/25,-28/25);
    \draw[very thick,blue,dashed] (-12/25,-18/25) .. controls (0.46,-0.06) .. (1/5,-1);
    \node[above right,fill=white,font=\small] at (9/25,-28/25) {$H_1$};
    \node[left,font=\small] at (0,1) {$n_1$};
    \node[right,font=\small] at (24/25,-7/25) {$n_1\times An_1$};
    \node[left,font=\small] at (-3/5,-4/5) {$ An_1$};
    \node[below left,blue,font=\small] at (0.18,-0.56) {$\mathcal{A}_{a,b}$};
\end{tikzpicture}
}
\subcaptionbox{$a\cdot n_1=2c_1$}{
\begin{tikzpicture}[scale = 2.1]
    \draw[->,thick] (0,0) -- (0,1);
    \draw[->,thick] (0,0) -- (24/25,-7/25);
    \draw[->,thick] (0,0) -- (-3/5,-4/5);
    \draw (-3/5,-4/5) -- (9/25,-28/25);
    \draw (24/25,-7/25) -- (9/25,-28/25);
    \draw[very thick,blue,dashed] (0.24,-0.07) -- (-0.36,-0.87);
    \draw[very thick,blue,dashed] (0.72,-0.21) -- (0.12,-1.01);
    \node[above right,fill=white,font=\small] at (9/25,-28/25) {$H_1$};
    \node[left,font=\small] at (0,1) {$n_1$};
    \node[right,font=\small] at (24/25,-7/25) {$n_1\times An_1$};
    \node[left,font=\small] at (-3/5,-4/5) {$ An_1$};
    \node[blue,font=\small] at (0.18,-0.56) {$\mathcal{A}_{a,b}$};
\end{tikzpicture}
}
\end{figure}
The following lines are dedicated to
further simplify \eqref{equation of xi,a,b},
which gives us the desired  estimate of the size of $\mathcal{A}_{a,b}\cap H_1\cap[-N,N]^3$. 
As in \eqref{definition of perp part}, we denote
\begin{equation}
    a^\perp = a - \frac{  a \cdot n_1 }{|n_1|^2} n_1 - \frac{a \cdot An_1}{|n_1|^2}An_1 \in \frac{1}{|n_1|^2} \mathbb Z^3.  \label{a perp part}
\end{equation}
There exists some $\eta \in \mathbb {Z}^3$ such that $|\eta| \leq |n_1\times An_1| \leq M^2$ and $2\xi^\perp - a^\perp = \frac{\lambda}{|n_1|^2} \eta$ for some $\lambda \in\mathbb{Z}$. Since the matrix $A$ is non-singular, $h(\eta)$ must be non-zero.
Now \eqref{equation of xi,a,b} can be written as
\begin{align}
   \frac{\lambda^2}{|n_1|^4} h(\eta) &=  h(2\xi^\perp-a^\perp) \\
   &= 2\left[ b+h(a^\perp)-h(a)+\frac{2c_1}{|n_1|^2}a\cdot n_1 \right] - \frac{8c_1-4a\cdot n_1}{|n_1|^2}\xi\cdot An_1.
\end{align}
Set $z = \lambda h(\eta)$, $y = |n_1|^2 \xi \cdot An_1$ and
\begin{align}
    q&= h(\eta )( 8c_1-4a_1\cdot n_1) , \\
    \omega&=2|n_1|^4h(\eta )\left[ b+h(a^\perp)-h(a)+\frac{2c_1}{|n_1|^2}a\cdot n_1 \right],
\end{align}
then $(y,z,q,\omega)$ is an integer solution of 
\begin{equation} 
z^2=qy+\omega, \quad |q|,|y|,|z| \lesssim M^8N. \label{eq:claim}\end{equation}

For given $a,b$, if $a \cdot n_1 \neq 2c_1$, which is equivalent to $a/2 \notin H_1$, then from \eqref{equation of xi,a,b}, $\xi\cdot An_1$ is determined by $\xi^\perp$.  From \eqref{definition of perp part} , $\xi\in \mathcal{A}_{a,b} \cap H_1$ is determined by $\xi^\perp$ since $\xi\cdot n_1=c_1$.
Let us state the following claim, postponing its proof for later.  

\textbf{Claim:} For given  $N,q\in\mathbb Z_+$ and $\omega \in \mathbb Z$,
  \begin{equation}
       \#\Big( \{ (y,z)\in \mathbb Z^2 \mid z^2= qy+\omega \} \cap [-N,N]^2\Big)\lesssim \sqrt{N}+\sqrt{q}.
  \end{equation}
Applying the claim to \eqref{eq:claim}, we see that  there are $O(M^{4}N^{1/2})$ many choices of $\xi^\perp$, hence 
\begin{equation}
    \#(\mathcal{A}_{a,b}\cap H_1\cap[-N,N]^3) \lesssim M^{4} N^{1/2}.\label{Aab-bound:curve}
\end{equation}

On the other hand, if $a\cdot n_1 =  2c_1$, i.e. $a/2\in H_1$, the left hand side of \eqref{equation of xi,a,b} is $0$ and we have at most two choices of $\xi^\perp$, hence from \eqref{definition of perp part} we know $\xi$ belongs to the union of two lines contained in $H_1$, 
\begin{equation}
    \#(\mathcal{A}_{a,b} \cap H_1\cap[-N,N]^3) \lesssim N. \label{Aab-bound:line}
\end{equation}

Now we prove the main estimate \eqref{est:prop-4f}.

\textbf{Case 1: }    
  $a\cdot n_{j_0} \neq 2c_{j_0}$, i.e. $a/2\notin H_{j_0}$, for some $1\leq j_0\leq 4$. Without loss of generality, we assume $j_0=1$. By Cauchy-Schwarz and \eqref{Aab-bound:curve} we have
    \begin{align}
       \Big( \sum_{\xi \in \mathcal{A}_{a,b} } f_{1}(\xi) f_{3}(a-\xi) \Big)^2 \lesssim M^{4} N^{1/2} \sum_{\xi\in\mathcal{A}_{a,b}} |f_{1}(\xi)|^2 | f_{3}(a-\xi)|^2,
    \end{align}
    summing over $a,b$ gives that
    \begin{equation}
        \sum_{\frac a2 \notin H_1} \sum_b \Big( \sum_{\xi \in \mathcal{A}_{a,b} } f_{1}(\xi) f_{3}(a-\xi) \Big)^2  \lesssim M^{4}N^{1/2} \| f_1 \|_{\ell^2(\mathbb Z^3)}^2 \| f_3 \|_{\ell^2(\mathbb Z^3)}^2.
    \end{equation}
    On the other hand 
    \begin{align}
\sum_{a,b}       \Big( \sum_{\xi \in \mathcal{A}_{a,b} } f_{2}(\xi) f_{4}(a-\xi) \Big)^2 &\lesssim \max\{ M^4N^{1/2}, N\} \sum_{a,b} \sum_{\xi\in\mathcal{A}_{a,b}} |f_{2}(\xi)|^2 | f_{4}(a-\xi)|^2 \\& 
\lesssim  \max\{M^4N^{1/2},N\} \| f_2 \|_{\ell^2(\mathbb Z^3)}^2 \| f_4 \|_{\ell^2(\mathbb Z^3)}^2,
    \end{align}
    as a result
\begin{align}
    \sum_{\frac a2 \notin H_1} \sum_b \sum_{\Gamma_{a,b}} f_1(\xi_1)f_2(\xi_2)f_3(\xi_3)f_4(\xi_4)
\lesssim M^{2}N^{1/2}(M^2+N^{1/4}) \prod_{j=1}^4 \| f_j \|_{\ell^2(\mathbb Z^3)}.
\end{align}

\textbf{Case 2: } $a/2 \in H_j $ for all $j$. 

\textbullet\ If $\dim \operatorname{span}_\mathbb{R} \{n_1,n_2,n_3,n_4\} = 3$. 
In this case there are at most one such $a$, and hence from \eqref{est:4f-Aab}
\begin{align}
   & \sum_{\frac{a}{2} \in \cap_j H_j} \sum_b \sum_{\Gamma_{a,b}} f_1(\xi_1)f_2(\xi_2)f_3(\xi_3)f_4(\xi_4)\\
    \lesssim &   \Big(\sum_{\xi_1} f_1(\xi_1)f_3(a-\xi_1)\Big)\Big( \sum_{\xi_2} f_2(\xi_2)f_4(a-\xi_2) \Big) 
    \leq \prod_{j=1}^4 \| f_j \|_{\ell^2(\mathbb Z^3)}.
\end{align}

\textbullet\ If $\dim \operatorname{span}_\mathbb{R} \{n_1,n_2,n_3,n_4\} = 2$.  
Let $2\leq j \leq 4$ be such that $$\operatorname{span}_\mathbb{R} \{ n_1,n_j\} = \operatorname{span}_{\mathbb{R}} \{ n_1,n_2,n_3,n_4 \}, $$ we will prove 
\begin{equation}
    \sum_{\frac a2\in H_1\cap H_j} \sum_{b} \Big( \sum_{\substack{\xi \in \mathcal{A}_{a,b} }} f_1(\xi)f_3(a-\xi) \Big)^2
    \lesssim \| f_1\|_{\ell^2(\mathbb Z^3)}^2 \| f_3 \|_{\ell^2(\mathbb Z^3)}^2.
\end{equation}
By symmetry, the same estimate holds for $f_2,f_4$.

Recall the definition of $a^\perp$ in \eqref{a perp part},
then we can write that
\begin{equation}
    2c_j = a\cdot n_j = \left[ a^\perp \cdot n_j + \frac{2c_1}{|n_1|^2}{ n_1\cdot n_j} \right] + \frac{a\cdot An_1}{|n_1|^2} An_1\cdot n_j.
\end{equation}
We claim that $An_1 \cdot n_j = n_1\cdot An_j\neq 0$. Otherwise, since $n_1,n_j \in \cone$, it holds that
\begin{equation}
    \operatorname{span}_{\mathbb{R}}\{An_1,An_j\} \perp  \operatorname{span}_{\mathbb{R}}\{n_1,n_j\},
\end{equation}
but their dimensions are both $2$
, which is a contradiction.
As a result, we can solve $a\cdot An_1$ in terms of $a^\perp$, and hence $a$ is determined by $a^\perp$.

 We set $\mathcal{A}_{a,b}^\perp = \{ \xi^\perp \mid \xi \in \mathcal{A}_{a,b} \}$, where $\xi^\perp$ is defined as \eqref{definition of perp part}. Then 
 \begin{equation}
  \mathcal{A}_{a,b} =\bigcup_{\beta\in \mathcal{A}_{a,b}^\perp}\{\xi \in \mathcal{A}_{a,b} \mid \xi^\perp=\beta\}  .
 \end{equation}
 From \eqref{equation of xi,a,b} we see that $\# \mathcal{A}_{a,b}^\perp \leq 2$ for each $a,b$.
    
By Cauchy-Schwarz inequality,
\begin{align}
     \Big( \sum_{\substack{\xi \in \mathcal{A}_{a,b} }} f_1(\xi)f_3(a-\xi) \Big)^2 &\lesssim \sum_{{\beta} \in \mathcal{A}_{a,b}^\perp}  \Big(\sum_{\substack{\xi\in [-N,N]^3,  \xi^\perp = {\beta}}} f_1(\xi)f_3(a-\xi) \Big)^2
     \\&\leq \sum_{{\beta} \in \mathcal{A}_{a,b}^\perp } {\Big( \sum_{\xi^\perp={\beta}} |f_1(\xi)|^2 \Big)} {\Big( \sum_{\xi^\perp = a^\perp - {\beta}} |f_3(\xi)|^2\Big)}.
\end{align}
Therefore,
\begin{align}
   & \mathrel{\phantom{\lesssim}} \sum_{\frac a2\in H_1\cap H_j} \sum_{b} \Big( \sum_{\substack{\xi \in \mathcal{A}_{a,b} }} f_1(\xi)f_3(a-\xi) \Big)^2 \\
    &\lesssim \sum_{\frac a2\in H_1\cap H_j} \sum_{b} \sum_{{\beta} \in \mathcal{A}_{a,b}^\perp } {\Big( \sum_{\xi^\perp={\beta}} |f_1(\xi)|^2 \Big)} {\Big( \sum_{\xi^\perp = a^\perp - {\beta}} |f_3(\xi)|^2\Big)}\\&
    =  \sum_{\frac a2\in H_1\cap H_j}  \sum_{{\beta}} {\Big( \sum_{\xi^\perp={\beta}} |f_1(\xi)|^2 \Big)} {\Big( \sum_{\xi^\perp = a^\perp - {\beta}} |f_3(\xi)|^2\Big)} \\&
    =  \| f_1\|_{\ell^2(\mathbb Z^3)}^2 \| f_3 \|_{\ell^2(\mathbb Z^3)}^2.
\end{align}

\textbullet\  If $\dim\operatorname{span}_\mathbb{R}\{ n_1,n_2,n_3,n_4\}=1$. In this case $n_1=n_2=n_3=n_4$ and $H_1=H_2=H_3=H_4$, otherwise the intersection is empty. Suppose ${(\xi_1,\xi_2,\xi_3,\xi_4)} \in \mathcal{Q}$ with $\xi_1-\xi_2$ is not a multiple of $An_1$, we have ${(\xi_1-\xi_2)\cdot A(\xi_1-\xi_4)=0}$ and $An_1\cdot A(\xi_1-\xi_4)=0$ since $n_1$ is the normal vector of $H_1$, which implies $A(\xi_1-\xi_4)$ is also a normal vector of $H_1$ and hence $A(\xi_1-\xi_4)$ is a multiple of $n_1$ and $\xi_1-\xi_4\in \cone$. As a result, we must have ${(\xi_1,\xi_2,\xi_3,\xi_4)}\in\mathcal{Q}_2$ and $\Omega_1{(f_1,f_2,f_3,f_4)}=0$.
    
\end{proof}

\begin{proof}[Proof of the claim]
   
    Denote  $
\wp_{q,\omega} = \{ (y,z)\in \mathbb Z^2 \mid z^2= qy+\omega \} \cap [-N,N]^2$ for fixed $N,q\in\mathbb Z_+$ and $\omega \in \mathbb Z$.
 Suppose $q = p_1^{\alpha_1}\dots p_r^{\alpha_r}$ is the prime factorization, we denote $\theta_i(z)$ the minimal residue of $z\pmod {p_i^{\alpha_i}}$ for $1\leq i\leq r$ and $\theta_0(z) = \lfloor z/q\rfloor$. Then the map
 \begin{equation}
     \mathbb Z \to \mathbb Z^{1+r},\quad z\mapsto (\theta_0(z),\theta_1(z),\dots,\theta_r(z))
 \end{equation}
 is an injection. In fact, if $\theta_i(z) = \theta_i(\tilde z )$ for all $1\leq i\leq r$, then $z-\tilde z$ is divided by all $p_i^{\alpha_i}$ and hence $z -\tilde z$ is divided by $q$. But $\theta_0(z)=\theta_0(\tilde z)$ implies $0\leq |z-\tilde z|\leq q-1$, which forces that $z=\tilde z$. As a result,
 \begin{equation}
     \# 
\wp_{q,\omega} \leq \prod_{i=0}^r \#\{ \theta_i(z) \mid (y,z) \in 
\wp_{q,\omega} \}.
 \end{equation}
Without loss of generality, we only consider the case $z>0$ and fix some $(y_0,z_0) \in 
\wp_{q,\omega}$.

We note that
$ \omega-qN \leq z^2 \leq \omega+qN$ for all $(y,z) \in 
\wp_{q,\omega}$.
If $\omega>2qN$, then $|z|,|z_0|>\sqrt{\omega/2}$ and hence
\[ \left|\theta_0(z) - \theta_0(z_0) \right| \leq 1+\frac{ |z^2- z_0^2| }{ q |z+z_0| }= 1+\frac{|y-y_0|}{|z+z_0|} \lesssim 1+ \frac{N}{\sqrt \omega} \lesssim 1+ \sqrt {N/q}. \]
If $\omega \leq 2qN$, then $|z| \lesssim \sqrt{qN}$ and hence $|\theta_0(z)| \lesssim \sqrt{ N/q}$.
Therefore, we know  $\theta_0(z)$ belongs to some interval of length $O(\sqrt{N/q}+1)$ for all $(y,z) \in 
\wp_{q,\omega}$.

On the other hand,
for any $(y,z) \in 
\wp_{q,\omega}$ we have $q|(z^2-z_0^2)$, and hence
 $p_i^{\alpha_i} | \theta_i(z-z_0)\theta_i(z+z_0)$ for $1\leq i\leq r$. Consequently, we can find some $\gamma_i\in\mathbb N$ which depends on $z$, such that $p^{\gamma_i}_i | \theta_i (z-z_0)$ and $p_i^{\alpha_i-\gamma_i} | \theta_i(z+z_0)$. This further implies
\begin{equation}
    \theta_i(z) \in \big( \theta_i(z_0) + p_i^{\gamma_i} \mathbb Z\big) \cap \big(-\theta_i(z_0)+p_i^{\alpha_i-\gamma_i}\mathbb{Z}\big).
\end{equation}
Let us put $\tilde\gamma_i =\max_{z} \min\{\gamma_i, \alpha_i-\gamma_i\}$. Notice $p_i^{\tilde\gamma_i} | 2\theta_i(z_0)$ and $p_i^{\max\{\gamma_i, \alpha_i-\gamma_i\}} \mathbb{Z}\subset p_i^{\alpha_i-\tilde\gamma_i}\mathbb{Z}$
, we see that at least one of 
\begin{equation}
    \theta_i(z)\in \big( \theta_i(z_0) + p_i^{\alpha_i-\tilde\gamma_i} \mathbb Z\big) ,\qquad \theta_i(z)\in  \big(-\theta_i(z_0)+p_i^{\alpha_i-\tilde\gamma_i}\mathbb{Z}\big) 
\end{equation}
holds true. Thus
\begin{equation}
    \{ \theta_i(z) \mid (y,z) \in 
\wp_{q,\omega} \} \subset \big( \theta_i(z_0) + p_i^{\alpha_i-\tilde\gamma_i} \mathbb Z\big) \cup \big(-\theta_i(z_0)+p_i^{\alpha_i-\tilde\gamma_i}\mathbb{Z}\big),
\end{equation}
and 
\begin{align}
    \#\{ \theta_i(z) \mid (y,z) \in 
\wp_{q,\omega} \}
    &\leq
    \begin{cases}
    2 p_i^{\tilde\gamma_i} , & \tilde\gamma_i\neq \alpha_i/2, \\
    p_i^{\alpha_i/2}, & \tilde\gamma_i=\alpha_i/2,
    \end{cases}
    \\
    &\leq p_i^{\alpha_i/2} \max \{1, 2p_i^{-1/2}\} .
\end{align}
 As a result $\#
\wp_{q,\omega}
 \lesssim \sqrt N+\sqrt{\vphantom{N} q}$.
\end{proof}

\subsection{Sharpness of Strichartz estimate}\label{subsec:examples}
We now present several examples showing the sharpness of \eqref{main-est}. 

Example 1: We take
\begin{equation}
    \phi_0(x) =  N^{-3/2} \sum_{\xi\in \mathbb{Z}^3 \cap [-N,N]^3} \me^{2\pi\mi \xi\cdot x}.
\end{equation}
It's easy to calculate that
$\| \phi_0\|_{{L^2_x(\mathbb T^3)}  }  \approx 1$,
while
 $   | \me^{\mi t\Box} \phi_0(x) | \gtrsim N^{3/2}$  for $ |x| < 1/{(100N)} $ and $ 0<t<1/{(100 N^2)}$.
As a consequence,
\begin{equation}
    \| \me^{\mi t \Box} \phi_0 \|_{{L_{t,x}^p([0,1]\times \mathbb T^3)}} \geq \left| \int_{\substack{|x|<\frac{1}{100 N},0<t<\frac{1}{100 N^2}}} | \me^{\mi t\Box} \phi_0|^p \dif t\dif x \right|^{1/p} \gtrsim N^{\frac 32 - \frac{5}{p} }.
\end{equation}
This example shows estimate~\eqref{main-est} is sharp for $p\geq 4.$

In particular for $p=4$, 
we consider $\Omega_2(\hat\phi_0)$, for each $\xi_1$, the number of choices of $\xi_2$ such that $\xi_1-\xi_2\in\cone$ is bounded by
\begin{align}
\#(\cone\cap[-N,N]^3) &\leq \sum_{M \leq N\text{ dyadic}} \#(\cone_M\setminus \cone_{M/2}) \\& \lesssim \sum_{M\leq N\text{ dyadic}} \frac{N}{M} \#(\cone_{M}^{\mathrm{irr}} \setminus \cone_{M/2}^{\mathrm{irr}}) \lesssim N\log N,
\end{align}
and $(\xi_3,\xi_4)$ lies on a plane passing through $\xi_1$ with normal vector $A(\xi_1-\xi_2)$, which gives $O(N^2)$ choices, hence
\begin{equation}
    \Omega_2(\hat\phi_0) \lesssim \log N .
\end{equation}
Thus $\Omega_1(\hat\phi_0)$ will give the major contribution in the $L^4$ estimate.

Example 2: We take
\begin{equation}
    \phi_0(x) = N^{-1/2}\sum_{\xi =1}^N \me^{2\pi\mi \xi x\cdot(1,1,0)}
\end{equation}
it's not hard to see
$    \| \phi_0\|_{{L^2_x(\mathbb T^3)}} \approx 1$.
Note that $\phi_0$ is invariant under the group $\{ \me^{\mi t\Box}\}_{t\in\mathbb R}$, hence
\begin{equation}
    |\me^{\mi t\Box} \phi_0(x)| = |\phi_0(x)|\gtrsim N^{1/2} \text{ for } |x \cdot (1,1,0)|<\frac{1}{100N},
\end{equation}
which implies
\begin{equation}
    \| \me^{\mi t\Box} \phi_0 \|_{{L_{t,x}^p([0,1]\times \mathbb T^3)} } \gtrsim N^{\frac12-\frac1p}.
\end{equation}
This example shows the estimate \eqref{main-est} is sharp for $p\in [2,4]$.  
Also for $p=4$, we notice $\Omega_1(\hat\phi_0)=0$, so $\Omega_2(\hat\phi_0)$ gives the major contribution.

Now if  we set $S=\operatorname{supp} \hat{\phi}_0$,  which is of size $N$, then we get 
\[ \| \me^{\mi t\Box} \phi_0 \|_{{L_{t,x}^4([0,1]\times \mathbb T^3)} }/\|\phi_0\|_{{{L^2_x(\mathbb T^3)}}}\gtrsim (\# S)^\frac14.\]
Consequently, for the $L^4$  estimate, we cannot obtain a non-trivial bound involving only $\#S$ without resorting to the trivial relation $\diam(S)\leq \#S$ (see Remark~\ref{rem:main-2}).

Example 3: We take
\begin{equation}
    \phi_{0,1}(y) = N^{-1/2} \sum_{\xi=1}^N \me^{2\pi\mi \xi y},\quad y\in \mathbb R
\end{equation}
and
\begin{equation}
    \phi_0(x) = N^{-1} \sum_{\xi,\eta=1}^N \me^{2\pi\mi (\xi,\xi,\eta)\cdot x} = \phi_{0,1}(x_1+x_2) \phi_{0,1}(x_3).
\end{equation}
Thus we see
 $   \| \phi_0 \|_{{L^2_x(\mathbb T^3)}} \approx 1$.
   
On the other hand, we have
\begin{align}
    | \me^{\mi t \Box}\phi_0(x)| &= |\phi_{0,1}(x_1+x_2) | |\me^{\mi t \partial^2}\phi_{0,1}(x_3)| \\& \gtrsim N^{1/2} |\me^{\mi{t}\partial^2} \phi_{0,1}(x_3)| 
\end{align}
for $|x_1+x_2|<1/(100N)$,
therefore
\begin{align}
    \| \me^{\mi t \Box} \phi_0 \|_{{L_{t,x}^p([0,1]\times \mathbb T^3)}} &\gtrsim N^{\frac12-\frac1p} \| \me^{\mi t\partial^2} \phi_{0,1} \|_{{L_{t,x}^p([0,1]\times \mathbb T^{})}}\\& \gtrsim N^{\frac12-\frac1p} \| \me^{\mi t \partial^2}{ \phi_{0,1}} \|_{{L_{t,x}^2([0,1]\times \mathbb T^{})}}\\& \approx N^{\frac12-\frac1p}.
\end{align}
This example also shows \eqref{main-est} is sharp for $p\in[2,4].$  We also observe for $p=4$, $\Omega_1(\hat\phi_0)=0$.

\begin{remark}
 We notice that the examples should easily generalize to the higher dimensional case. Example 1 also extends to irrational tori, while the validity of examples 2 and 3 on irrational tori depends on the equation.

\end{remark}

\section{Local Well-posedness}\label{Section: Local Well-posedness}

\subsection{Function spaces}
We use the adapted function spaces $X^s,Y^s$, whose definitions are based on the $U^p,V^p$ spaces.  We will give their definitions and state the basic properties. We refer the readers to \cite{hadac2009,herr2011} for detailed proofs of the following propositions, where a general theory can also be found.

Let $\mathcal H$ be a separable Hilbert space over $\mathbb C$; in this paper, this will be $\mathbb C$ or $H^s(\mathbb T^3)$.
Let $\mathcal Z$ be the set of finite partitions $-\infty <t_0 <t_1 < \cdots < t_K \leq \infty$ of the real line. 
\begin{defi} 
Let $1 \leq p < \infty$. For $\{ t_k \}_{k=0}^K \in\mathcal Z$ and $ \{ \phi_k \}_{k=0}^{K-1} \subset \mathcal H$ with $\sum_{k=1}^{K-1} \| \phi_k \|_{\mathcal H}^p =1$, we call a piecewise defined function $a:\mathbb R \to \mathcal H$,
\[ a(t) = \sum_{k=1}^{K-1} \chi_{[t_k,t_{k+1})} \phi_k \]
a $U^p$-atom, and we define the atomic space $U^p(\mathbb R, \mathcal H)$ of all functions $u \colon \mathbb R\to\mathcal H$ such that
\[ u = \sum_j \lambda_j a_j ,\quad \mbox{ with } a_j \mbox{ are } U^p\mbox{-atoms, and  } \{\lambda_j\} \in \ell^1,\]
with norm
\[
 \| u\|_{U^p(\mathbb R,\mathcal H)} := \inf \left\{ \sum_j |\lambda_j|  \Biggm| u = \sum_j \lambda_j a_j, \  a_j \text{ are } U^p\text{-atoms}\right\}.
\]
\end{defi}

\begin{defi}
Let $1\leq p<\infty$, we define the space $V^p(\mathbb R,\mathcal H)$ of functions $v \colon \mathbb R \to \mathcal H$ such that $\lim_{t\to-\infty} v(t)=0$ and the norm
\[
 \| v \|_{V^p(\mathbb R, \mathcal H)} := \sup_{\{t_k\}_{k=0}^K \in\mathcal Z} \left( \sum_{k=0}^{K-1} \| v(t_{k+1}) - v(t_k) \|_{\mathcal{H}}^p \right)^{1/p}
\]
is finite.
\end{defi}

Corresponding to the linear  flow generated by the group $\{\me^{\mi t\Box}\}_{t\in\mathbb R}$, we define the following.

\begin{defi}
For $s\in \mathbb R$, we define the space $U^p_\Box H^s$ $(\mbox{resp., } V^p_\Box H^s)$ of functions $u \colon \mathbb R \to H^s(\mathbb T^3)$ such that $ t \mapsto \me^{-\mi t\Box} u(t)$ is in $U^p(\mathbb R, H^s(\mathbb T^3))$ $(\mbox{resp., } V^p(\mathbb R, H^s(\mathbb T^3)))$ with the norms
\[ \| u \|_{U^p_\Box H^s} := \|  \me^{-\mi t\Box} u  \|_{U^p(\mathbb R, H^s(\mathbb T^3))} ,\quad  \| u\|_{V^p_\Box H^s} := \| \me^{-\mi t\Box} u \|_{V^p(\mathbb R, H^s(\mathbb T^3))}  . \]
\end{defi}
Due to the atomic structure of $U^p$, we can extend bounded operators on $L^2(\mathbb T^3)$ to $U^p_\Box L^2$.
\begin{prop} [{\cite[Proposition 2.19]{hadac2009}}] \label{extension}
Let $1\leq p< \infty$ and $T_0 \colon L^2(\mathbb T^3) \times \cdots \times L^2(\mathbb T^3) \to L_{\mathrm{loc}}^1 ( \mathbb R \times \mathbb T^3)$ be a $n$-linear operator. If 
\[ \| T_0( \me^{\mi t\Box} \phi_1, \cdots , \me^{\mi t\Box} \phi_n) \|_{L^p_{t,x}} \leq C_{T_0} \prod_{i=1}^n \| \phi_i \|_{{L^2_x(\mathbb T^3)}}, \]
then $T_0$ extends to a $n$-linear operator $T$ on $U^p_\Box L^2 \times \cdots \times U_\Box^pL^2$, satisfying
\[ \| T ( u_1, \cdots , u_n) \|_{L^p_{t,x}} \lesssim C_{T_0} \prod_{i=1}^n \| u_i \|_{U^p_\Box L^2}. \]
\end{prop}
The following corollary is a direct application of this proposition to our main result Theorem \ref{sharp T3 L4 hyperbolic Strichartz} and Remark~\ref{rem:main-2}.
\begin{coro}
For $u\in U^4_\Box L^2$, and any cube $C$ of side length $N$, we have 
\[ \| P_{C} u \|_{L^4_{t,x}([0,1]\times \mathbb T^3)} \lesssim N^{1/4} \| u \|_{U^4_\Box L^2}. \label{Strichartz estimate in U^4_Box}\]
\end{coro}

\begin{defi}
For $s\in\mathbb R$, we define the space $X^s$ of functions $u\colon \mathbb R \to H^s(\mathbb T^3)$ such that for every $\xi \in \mathbb Z^3$ the mapping $ t \mapsto \me^{ -\mi t h(\xi)} \widehat{u(t)}(\xi)$ is in $U^2(\mathbb R, \mathbb C ) $, with the norm
\[ \|u\|_{X^s} := \left( \sum_{\xi \in \mathbb Z^3} \left<\xi\right>^{2s} \| \me^{ -\mi t h(\xi)} \widehat{u(t)}(\xi) \|_{U^2(\mathbb R, \mathbb C)}^2 \right)^{1/2} .\]
\end{defi}
\begin{defi}
For $s\in\mathbb R$, we define the space $Y^s$ of functions $u\colon \mathbb R \to H^s(\mathbb T^3)$ such that for every $\xi\in \mathbb Z^3$ the mapping $ t \mapsto \me^{ -\mi t h(\xi)} \widehat{u(t)}(\xi)$ is in $V^2(\mathbb R,  \mathbb C ) $, with the norm
\[ \|u\|_{Y^s} := \left( \sum_{\xi\in \mathbb Z^3} \left<\xi\right>^{2s} \| \me^{ -\mi t h(\xi)} \widehat{u(t)}(\xi) \|_{V^2(\mathbb R, \mathbb C)}^2 \right)^{1/2} .\]
\end{defi}

\begin{remark}\label{rem:embedding}
We have the embeddings\[U^2_\Box H^s \hookrightarrow X^s \hookrightarrow Y^s \hookrightarrow V^2_\Box H^s \hookrightarrow U^q_\Box H^s\hookrightarrow L^\infty H^s,\quad \forall q\in(2,\infty).\]
\end{remark}
\begin{remark}\label{remark}
For $s\in \mathbb R$, and $S_1,S_2$ are disjoint subsets of $\mathbb Z^3$, we have
\[ \| P_{S_1\cup S_2} u\|_{Y^s} ^2 = \| P_{S_1} u\|_{Y^s}^2 + \|P_{S_2} u\|_{Y^s}^2. \]
\end{remark}
For time interval $I\subset \mathbb R$, we also consider the restriction spaces $X^s(I),Y^s(I)$ with norms
\[
\| u \|_{X^s(I)} = \inf \{ \| \tilde{u} \|_{X^s} \mid \tilde{u}|_{I}=u \},\quad
\| u \|_{Y^s(I)} = \inf \{ \| \tilde{u} \|_{Y^s} \mid \tilde{u}|_{I}=u \}.
\]
\begin{prop}[{\cite[Proposition 2.10]{herr2011}}]
Let $s\in \mathbb{R}$ and $T>0$. For $\phi \in H^s(\mathbb T^3)$, we have $\me^{\mi t\Box} \phi \in X^s([0,T))$ and
\[ \| \me^{\mi t\Box} \phi \|_{X^s([0,T))} \leq \| \phi \|_{H^s(\mathbb T^3)} .\]
For $f\in L^1 ([0,T); H^s(\mathbb T^3))$, we have  the estimate for the Duhamel term.
\begin{multline*}
\left\| \int_0^t \me^{\mi (t-t')\Box} f(t') \dif t' \right\|_{X^s([0,T))} \leq 
\sup_{\substack{v\in Y^{-s}([0,T)) \\ \|v\|_{Y^{-s}([0,T))}\leq 1}  } \left| \iint_{[0,T)\times \mathbb T^3} f(t,x) \overline{v(t,x)} \dif x\dif t \right|
\end{multline*}
\end{prop}

\begin{remark}
The $X^s([0,T))$ norm of the Duhamel term is also controlled by $ \| f \|_{L^1([0,T);H^s(\mathbb T^3))}$.
\end{remark}
\subsection{Multilinear estimates} We start from a bilinear estimate for frequency localized functions on $\mathbb T^3$.
\begin{prop}\label{Bilinear Strichartz estimate on T^3}
    For $u_1,u_2 \in Y^0([0,1])$ with $u_i = P_{N_i} u_i$, we have
    \begin{equation}
        \| u_1 u_2 \|_{{L_{t,x}^2([0,1]\times \mathbb T^3)}} \lesssim \min\{ N_1,N_2 \}^{1/2} \| u_1\|_{Y^0([0,1])} \|u_2\|_{Y^0([0,1])}.
    \end{equation}
\end{prop}
\begin{proof}
    We may assume that $N_1\leq N_2$. We decompose $\mathbb Z^3 = \bigcup_j  C_j$ into almost disjoint cubes with side length $N_1$ and write
    \begin{equation}
        u_1u_2 = \sum_{C_j} u_1 P_{C_j} u_2.
    \end{equation}
    Their Fourier supports are finitely overlapped, hence we have the almost orthogonality
    \begin{align}
        \| u_1u_2\|_{{L_{t,x}^2([0,1]\times \mathbb T^3)}}^2 &\approx \sum_j \| u_1 P_{C_j} u_2\|_{{L_{t,x}^2([0,1]\times \mathbb T^3)}}^2 \\&
        \leq \sum_j \| u_1 \|_{{L_{t,x}^4([0,1]\times \mathbb T^3)}}^2 \| P_{C_j} u_2 \|_{{L_{t,x}^4([0,1]\times \mathbb T^3)}}^2. 
    \end{align}
    { By Corollary \ref{Strichartz estimate in U^4_Box}, Remark \ref{remark} and the embedding properties in  Remark \ref{rem:embedding},}
    \begin{align}
        \| u_1u_2\|_{{L_{t,x}^2([0,1]\times \mathbb T^3)}}^2 &\lesssim{} \sum_j N_1 \| u_1\|_{Y^0([0,1])}^2 \| P_{C_j} u_2 \|_{Y^0([0,1])}^2 \\&
        ={} N_1\| u_1\|_{Y^0([0,1])}^2 \|  u_2 \|_{Y^0([0,1])}^2.
    \end{align}
  
\end{proof}

Now we are ready to show the key estimate on the nonlinear term by using duality argument combined with frequency decomposition, which helps to treat the nonlinearity in the fixed point argument.

\begin{prop}\label{k-multilinear estimate on T^3}
    Let $k\in \mathbb{N}_+$, $s=\frac 32 - \frac 1k$ if $k\geq 2$ and $s>\frac12$ if $k=1$. Then for any $0<T<1$,
    for $u_1,\dots,u_{2k+1} \in {X^{s}([0,T))}$, we have
    \begin{equation}
        \left\| \int_0^t \me^{\mi (t-t')\Box} \prod_{i=1}^{2k+1} u_i \dif t' \right\|_{{X^{s}([0,T))}} \lesssim_{s,k} \prod_{i=1}^{2k+1} \| u_i\|_{{X^{s}([0,T))}}.
    \end{equation}
    Here the  implicit constant does not depend on T.
\end{prop}

\begin{proof}
    It suffices to show that for any $u_0 \in Y^{-s}([0,T))$, we have
    \begin{equation}
        \left| \int_{[0,T)\times \mathbb T^3} u_0\prod_{i=1}^{2k+1}u_i \dif x\dif t \right| \lesssim \| u_0\|_{Y^{-s}([0,T))} \prod_{i=1}^{2k+1} \| u_i \|_{X^{s}([0,T))}.
    \end{equation}
    We apply Littlewood-Paley decomposition to each $u_i$ to write
    \begin{equation}
        u_i = \sum_{N_i\text{ dyadic}} P_{N_i}u_i = \sum_{N_i\text{ dyadic}} u_{N_{i}}^{(i)},
    \end{equation}
    hence it suffices to estimate
    \begin{equation}
        \sum_{N_0,\dots,N_{2k+1}} \left| \int_{[0,T)\times \mathbb T^3} u_{N_{0}}^{(0)}\prod_{i=1}^{2k+1} u_{N_{i}}^{(i)} \dif x\dif t \right|.
    \end{equation}
    In order to make the integral non-zero, we must have that the two highest frequencies are comparable. Due to symmetry, it's harmless to assume $N_1\geq N_2\geq \dots\geq N_{2k+1}$. Following Proposition \ref{Bilinear Strichartz estimate on T^3} we have that
    \begin{align}
        &\left| \int_{[0,T)\times \mathbb T^3} u_{N_{0}}^{(0)}\prod_{i=1}^{2k+1} u_{N_{i}}^{(i)} \dif x\dif t \right| \\\leq&{} \| 
        u_{N_{0}}^{(0)}u_{N_{2}}^{(2)}\|_{{L_{t,x}^2([0,1]\times \mathbb T^3)}} \| u_{N_{1}}^{(1)} u_{N_{3}}^{(3)} \|_{{L_{t,x}^2([0,1]\times \mathbb T^3)}} \prod_{i\geq 4} \| u_{N_{i}}^{(i)} \|_{{L_{t,x}^\infty([0,1]\times \mathbb T^3)}} \\
        \lesssim&{} \min\{ N_0,N_2\}^{\frac12}_{\vphantom{3}} N_3^{\frac12} \| u_{N_{0}}^{(0)}\|_{Y^0} \| u_{N_{1}}^{(1)}\|_{Y^0} \| u_{N_{2}}^{(2)}\|_{Y^0}\| u_{N_{3}}^{(3)}\|_{Y^0} \prod_{i\geq 4} N_i^{3/2} \| u_{N_{i}}^{(i)} \|_{Y^0} \\
        \approx&{} \min\{ N_0,N_2\}^{\frac12}_{\vphantom{3}} \frac{N_0^sN_3^{\frac12-s}}{N_1^{s}N_2^s}  \| u_{N_{0}}^{(0)}\|_{Y^{-s}} \| u_{N_{1}}^{(1)}\|_{Y^s} \| u_{N_{2}}^{(2)}\|_{Y^s}\| u_{N_{3}}^{(3)}\|_{Y^s} \prod_{i\geq 4} N_i^{\frac 32-s} \| u_{N_{i}}^{(i)} \|_{Y^s}.
    \end{align}
        For $k=1$, since $s>1/2$, we directly apply Cauchy-Schwarz to the summation over the two lower frequencies and the two highest frequencies respectively to the desired conclusion.
    For $k\geq 2$, applying Cauchy-Schwarz to summation over $N_i$ for $i\geq 4$, we get that
    \begin{align}
        &\min\{ N_0,N_2\}^{\frac12}_{\vphantom{3}} \frac{N_0^sN_3^{s-\frac12}}{N_1^{s}N_2^s}  \| u_{N_{0}}^{(0)}\|_{Y^{-s}} \| u_{N_{1}}^{(1)}\|_{Y^s} \| u_{N_{2}}^{(2)}\|_{Y^s}\| u_{N_{3}}^{(3)}\|_{Y^s} \prod_{i\geq 4} \| u_i\|_{{X^{s}([0,T))}} \\
        \leq{}& \left( \frac{N_0}{N_1} \right)^{s} \left( \frac{ N_3 }{N_2} \right)^{s-\frac12}\| u_{N_{0}}^{(0)}\|_{Y^{-s}} \| u_{N_{1}}^{(1)}\|_{Y^s} \| u_{N_{2}}^{(2)}\|_{Y^s}\| u_{N_{3}}^{(3)}\|_{Y^s} \prod_{i\geq 4} \| u_i\|_{{X^{s}([0,T))}}.
    \end{align}
    Then apply Cauchy-Schwarz to the summation over the two lower frequencies and the two highest frequencies respectively, we get the desired conclusion.

\end{proof}

\subsection{Proof of Theorem \ref{local well-posedness of HNLS on T^3}} The proof is a standard contraction argument as in~\cite{herr2011, Killip2016}. 
Given initial data $\phi \in H^s(\mathbb T^3)$, with $\| \phi \|_{H^s(\mathbb T^3)} \leq A$, suppose $\delta$ is a small constant depending on $A$, and $N$ is a large number depending on $\phi$ and $\delta$ such that $\|P_{>N}\phi \|_{H^s(\mathbb T^3)} \leq \delta$, we will show the Picard iteration mapping given by
\[
\mathcal I(u)(t) := \me^{\mi t\Box}\phi \mp \mi \int_0^t \me^{\mi (t-t') \Box} |u|^{2k}u \dif t'.
\]
is a contraction on the set 
\begin{multline*}
D : = \{ u \in C([0,T);H^s(\mathbb T^3)) \cap X^s([0,T)) \mid \\
u(0)=\phi,\ \|u\|_{X^s([0,T))} \leq 2A,\ \|  P_{>N} u \|_{X^s([0,T))} \leq 2\delta
\},
\end{multline*}
under the metric
\[ d(u,v) := \| u-v \|_{X^s([0,T))} \]
 provided $T$ is chosen sufficiently small (depending on $A$, $\delta$,  $N$ and $k$). 
 
For $u,v\in D$, we can decompose 
\begin{equation}
    |u|^{2k}u-|v|^{2k}v = F_1(u,v) + F_2(u,v),
\end{equation}
where $F_1(u,v)$ is a combination of $u-v,P_{\leq N}u,P_{\leq N}v$, and all terms involving $P_{>N}u,P_{>N}v$ appear in $F_2(u,v)$.
Employing Sobolev embeddings and  \cite[Theorem A.12]{Kenig1993}, we estimate that
\begin{align}
   & \left\| \int_0^t \me^{\mi (t-t')\Box} F_1(u,v) \dif t' \right\|_{X^{s}([0,T))} \leq CT \| F_1(u,v)\|_{L^\infty H^s} \\
    \leq{}& CT \left( \| u-v \|_{L^\infty H^s} \left(\| P_{\leq N} u\|_{L^\infty_{t,x}}^{2k} + \| P_{\leq N} v \|_{L^\infty_{t,x}}^{2k}\right) \right.\\
   &\qquad + N^s \| u-v\|_{L_t^{\infty\vphantom{/}} L_{x\vphantom{t}}^{6/(3-2s)}} \left.\left (\| P_{\leq N} u\|_{L_t^{\infty\vphantom{/}} L_{x\vphantom{t}}^{6k/s}}^{2k}+ \| P_{\leq N} v \|_{L_t^{\infty\vphantom{/}} L_{x\vphantom{t}}^{6k/s}}^{2k} \right)\right) \\
    \leq{}& CTN^{k(3-2s)} (2A)^{2k}\| u-v \|_{{X^{s}([0,T))}} .\label{est:contraction-low}
\end{align}
While by Proposition \ref{k-multilinear estimate on T^3}, it holds that
\begin{align}
   & \left\| \int_0^t \me^{\mi (t-t')\Box} F_2(u,v) \dif t' \right\|_{X^{s}([0,T))} \\ \leq {}&
    C \| u-v\|_{X^s} ( \| P_{>N} u \|_{X^s} + \| P_{>N}v \|_{X^s} )( \| u \|_{X^s} + \| v \|_{X^s} )^{2k-1} \\
    \leq{}& C(2A)^{2k-1}(2\delta) \| u-v\|_{X^{s}([0,T))}. \label{est:contraction-high}
\end{align}
Hence we get that 
\begin{equation}
\left\| \mathcal{I}(u) -  \mathcal{I}(v) \right\|_{X^s} 
\leq  \frac{1}{10} \| u-v\|_{X^s}, \label{contraction on T^3}
\end{equation}
provided $\delta$ is chosen sufficiently small depending on $A, k$, and $T$ is chosen sufficiently small depending on $A,N$ and $k$.

Next we verify that $\mathcal{I}$ maps $D$ into itself. For constant $C$ large enough, we have
\begin{align}
    \left\| \int_0^t \me^{\mi (t-t') \Box } |P_{\leq N} u |^{2k} P_{\leq N} u \dif t' \right\|_{X^s} & \leq CT \left\| |P_{\leq N} u |^{2k} P_{\leq N} u \right\| _{L^\infty H^s} \\&
    \leq C T \| P_{\leq N} u \|_{L^\infty_{t,x}}^{2k} \| P_{\leq N }u\|_{X^s} \\& \leq C T N^{k(3-2s)}(2A)^{2k+1} ,\label{est:Duhamel low}
\end{align}
and apply \eqref{contraction on T^3} for $v=P_{\leq N}u$ to get that
\begin{align}
    \| \mathcal{I}(u) - \mathcal{I}(P_{\leq N}u) \|_{X^s} \leq \frac{1}{10} \| P_{>N} u \|_{X^s} \leq \frac{\delta}{5}.\label{est:diffwithlow}
\end{align}
To control $P_{>N} \mathcal{I}(u)$, notice at least one input in the nonlinear term should have high frequency $\frac{N}{2k+1}$, thus applying \eqref{est:Duhamel low}\eqref{est:diffwithlow} we get 
\begin{align}
\| P_{>N} \mathcal{I}(u)\|_{X^s}\lesssim \|P_{>N}\me^{\mi t\Box}\phi\|_{X^s}+  \|P_{>N}\Big(\mathcal{I}(u) - \mathcal{I}(P_{\leq \frac{N}{2k+1}}u)\Big) \|_{X^s}\leq 2\delta.
\end{align}

To summarize, 
provided $\delta$ is chosen sufficiently small depending on $A, k$, and $T$ is chosen sufficiently small depending on $A, N$ and $k$, we have
\[ \| \mathcal I(u)\|_{X^s} \leq \| \me^{\mi t\Box}\phi \|_{X^s} +A\leq 2A,\quad  \| P_{>N} \mathcal I(u)\|_{X^s} \leq  \|P_{>N} \me^{\mi t\Box}\phi \|_{X^s} +\delta \leq 2\delta.\]

 As for the uniqueness  in the whole space $C([0,T);H^s(\mathbb T^3)) \cap X^s([0,T))$, supposing that we have two functions $u,v$ which both solve the equations \eqref{HNLS on T3} with the same initial data $\phi$, we can choose $A'$ sufficiently large, $\delta'$ sufficiently small and $N'$ sufficiently large such that $u,v$ are both contained in some $D=D_{A',N',\delta'}$. By the iteration, we know that there exists some $T'$ (maybe much smaller than $T$ given above) such that $u(t)=v(t)$ for $t\in[0,T')$.
 Uniqueness in the whole space $C([0,T);H^s(\mathbb T^3)) \cap X^s([0,T))$ follows from a continuity argument.

\subsection{Proof of Theorem \ref{ill-posedness of HNLS on T^3}}
We prove the ill-posedness of the cubic HNLS on $H^{1/2}(\mathbb T^3)$ by showing the first Picard iteration is unbounded.  Let us   pick
\begin{equation}
    \phi_N(x) = \sum_{k=1}^N \frac{ \me^{2\pi\mi(k,k,0)\cdot x}}{k}.
\end{equation}
It is easy to see $\| \phi_N \|_{H^{1/2}(\mathbb T^3)} \approx (\log N)^{1/2}$. Notice that $\Box \phi_N = 0$, so $\phi_N$ (also $|\phi_N|^{2}\phi_N$) is invariant under the group $\{ \me^{\mi t\Box} \}_{t\in\mathbb R}$, thus
\begin{equation}
    \mathcal{I}(\me^{\mi t\Box} \phi_N) (t) =  \mathcal{I}(\phi_N)(t) = \phi_N \pm \mi t |\phi_N|^{2}\phi_N. \label{1st-iteration}
\end{equation}
It suffices to show that 
$  \| |\phi_N|^2 \phi_N \|_{H^{1/2}(\mathbb T^3)} \gtrsim \log N \| \phi_N \|^3_{H^{1/2}(\mathbb T^3)}  $. Since
\begin{equation}
    | \phi_N|^2\phi_N(x) = \sum_k \me^{2\pi\mi (k,k,0)\cdot x} \sum_{k_1-k_2+k_3=k} \frac{1}{k_1k_2k_3},
\end{equation}
we consider the set 
\[ 
\Gamma(k) = \{ (k_1,k_2,k_3) \in \mathbb{Z}^3 \mid k_3= k-k_1+k_2,\ 1\leq k_1,k_2 \leq k/4 \}
\]
 for $k$ positive and sufficiently large. Then $k/2 \leq k_3 \leq 3k/2$ for $(k_1,k_2,k_3)\in \Gamma(k)$. Hence
\[ \sum_{\substack{{k_1-k_2+k_3=k}\\ 1\leq k_i \leq N}} \frac{1}{k_1k_2k_3} \geq \sum_{\Gamma(k)}\frac{1}{k_1k_2k_3} \approx \frac{1}{k} \sum_{k_1=1}^{k/4} \frac{1}{k_1} \sum_{k_2=1}^{k/4}\frac{1}{k_2} \approx \frac{(\log k)^2}{k}, \]
for $1\lesssim k\lesssim N$, and
\[ \Big\| \phi_N |\phi_N|^2 \Big\|_{H^{1/2}(\mathbb{T}^2)} \gtrsim \left( \sum_{1\lesssim k\lesssim N} k\cdot \frac{(\log k)^4}{k^2}\right)^{1/2} \approx (\log N)^{5/2}\approx \log N \| \phi_N \|^3_{H^{1/2}(\mathbb T^3)}, \]
 this finishes the proof.~\footnote{
The construction presented here is essentially two-dimensional. Therefore, it can also be used to  prove the ill-posedness of the 2D cubic HNLS for initial data in $H^{1/2}(\mathbb{T}^2)$, in the same sense as Theorem~\ref{ill-posedness of HNLS on T^3}.}

\bibliographystyle{plain}
\bibliography{_ref}

\end{document}